\documentclass[11pt]{article}
\usepackage{xcolor}

\usepackage[title]{appendix}

\usepackage{pst-all}
\usepackage{pgfplots}
\usepackage{pst-plot}
\usetikzlibrary{calc}

\usepackage{xcolor}
\usepackage{tikz}
\usepackage{tkz-euclide}

\usepackage{bbm}
\usepackage{hyperref}
\hypersetup{
    colorlinks=true,              
    breaklinks=true,              
    urlcolor= black,              
    linkcolor= blue,              
    bookmarksopen=false,
    filecolor=black,
    citecolor=blue,
    linkbordercolor=blue}

\usepackage{chngcntr}

\usepackage{mathrsfs}

\usepackage[margin=1in]{geometry}

\usepackage{amsmath,amsthm,amssymb}

\usepackage{bbm}

\usepackage{graphicx}
\usepackage{faktor}

\newcommand{\R}{\mathbb{R}}

\newcommand{\N}{\mathbb{N}}

\newcommand{\Z}{\mathbb{Z}}

\newcommand{\cov}{\normalfont \text{Cov}}
\newcommand{\vol}{\normalfont \text{Vol}}

\newcommand{\End}{\normalfont \text{End}}

\newcommand{\pr}{\normalfont \text{pr}}

\newcommand{\GL}{\normalfont \text{GL}}

\newcommand{\SL}{\normalfont \text{SL}}

\newcommand{\SO}{\normalfont \text{SO}}

\newcommand{\Ad}{\normalfont \text{Ad}}

\newcommand{\Supp}{\normalfont \text{Supp}}

\newcommand{\s}{\normalfont \textbf{s}}

\newcommand{\st}{\sqrt{t}}

\newcommand{\e}{\epsilon}

\newcommand{\de}{\delta}

\newcommand{\La}{\Lambda}

\newcommand{\ga}{\gamma}
\newcommand{\Ga}{\Gamma}

\newcommand{\Sph}{\mathbb{S}}

\newcommand{\Span}{\normalfont \text{Span}}

\newtheorem{theorem}{Theorem}[section]

\newtheorem{lemma}[theorem]{Lemma}

\newtheorem{proposition}[theorem]{Proposition}

\newtheorem{corollary}[theorem]{Corollary}

\newtheorem{claim}{Claim}

\theoremstyle{definition}

\theoremstyle{remark}

\newtheorem{remark}[theorem]{Remark}

\numberwithin{equation}{section}

\usepackage{setspace}

\usepackage[
style=alphabetic,
sorting=nty,
maxnames=99,
maxalphanames=99
]{biblatex}
\renewbibmacro{in:}{}
\DeclareFieldFormat[article]{citetitle}{\mkbibemph{#1}}
\DeclareFieldFormat[article]{title}{\mkbibemph{#1}}
\DeclareFieldFormat[article]{journaltitle}{#1}

\usepackage{enumerate}

\AtEndDocument{%
  \par
  \medskip
    \begin{tabular}{@{}l@{}}%
    \textsc{Department of mathematics, the Ohio State University}\\
    \vspace{0.5cm}
    \textit{E-mail address}: \texttt{bersudsky.1@osu.edu}\\
    
    \textsc{Department of mathematics, the Ohio State University}\\
    \textit{E-mail address}: \texttt{xing.211@osu.edu}
  \end{tabular}}
  
\title{Equidistribution of lattice orbits in the space of homothety classes of rank $2$ sublattices in $\R^3$}
\author{Michael Bersudsky and Hao Xing}

\begin{document}
\date{}
\maketitle
\begin{abstract}
We study the distribution of orbits of a lattice $\Gamma\leq\SL(3,\R)$ in the moduli space $X_{2,3}$ of covolume one rank-two discrete subgroups in $\R^3$. Each orbit is dense, and our main result is the limiting distribution of these orbits with respect to norm balls, where the norm is given by the sum of squares. Specifically, we consider  $\Gamma_T=\{\gamma\in\Gamma:\|\gamma\|\leq T\}$ and show that, for any fixed $x_0\in X_{2,3}$ and $\varphi\in C_c(X_{2,3})$, 
$$\lim_{T\to\infty}\frac{1}{\#\Gamma_T}\sum_{\gamma\in\Gamma_T}\varphi(x_0\cdot\gamma)=\int_{X_{2,3}}\varphi(x)d \tilde\nu_{x_0}(x),$$
where $\tilde\nu_{x_0}$ is an explicit probability measure on $X_{2,3}$ depending on $x_0$. To prove our result, we use the duality principle developed by Gorodnik and Weiss which recasts the above problem into the problem of computation of certain volume estimates of growing skewed balls in $H$ and proving ergodic theorems of the left action of the skewed balls on $\SL(3,\mathbb{R})/\Gamma$. The ergodic theorems are proven by applying theorems of Shah building on the linearisation technique. The main contribution of the paper is the application of the duality principle in the case where $H$ has infinitely many non-compact connected components.

\end{abstract}

\section{Introduction}





\subsection{The problem set-up and statement of the main result}


In this paper we study the asymptotic distributional properties of the action of lattice subgroups of the special linear group on a moduli space of discrete subgroups of the Euclidean space. Such a research direction is in particular motivated by \cite{Sargent2017DynamicsOT} and the more recent work \cite{gorodnik2022stationary}  studying random walks on such spaces. We will study the dynamics with respect to the action of growing norm balls.

We proceed to describe our specific setting. We say that $\La\subset\R^3$ is a 2-lattice, if $\La$ is the $\Z$-Span of a tuple of linearly independent vectors  ${v_1,v_2}\in\R^3$, that is, $$\Lambda:=\text{Span}_{\Z}\{v_1,v_2\}.$$ For $\La$ we let $$\cov(\Lambda):=\sqrt{\det(\langle v_i,v_j\rangle)},$$
which is the area of a fundamental parallelogram of $\Lambda$. 
The \emph{unimodular} 2-lattices $\La\subset\R^3$ are those $\La$ with $\cov(\La)=1$. 
 Next, we recall the definition of the shape of $\La$, a notion that was extensively studied in e.g \cite{emss, AES_lat}, which refined the classical work of Schmidt \cite{Schmidt1998TheDO}.
We view $\R^3$ as row vectors, and for a unimodular 2-lattice $\La\subset\R^3$, let $w\in\mathbb{S}^2$ such that $w\perp \La$. We choose $\rho \in \SO(3,\R)$ such that $w \rho=e_3:=(0,0,1)$, and we define 
\begin{equation}\label{definition of shape}
 \s(\La,w):=\La \rho\begin{bmatrix}
\SO(2,\R) & 0 \\
0 &  1 
\end{bmatrix}, 
\end{equation}
which is independent of the choice of $\rho$. By identifying the unimodular 2-lattices $\La\subset {e}_3^\perp$ with $X_2$, we identify $\s(\Lambda,w)$ as a point in $X_2\slash \SO(2,\R)$. Note that in general $$\s(\La,w)\neq\s(\La,-w).$$
\begin{remark}
 A more intrinsic definition of a shape of a discrete subgroup, see e.g. \cite{Sargent2017DynamicsOT, Schmidt1998TheDO}, is defined by the equivalence class under the equivelence relation of scaling and rotations. When defining shape in this way, one gets a point in $X_2/\text{O}(2,\R)$, and it captures slightly less information. Our definition \eqref{definition of shape} is mainly motivated by the definition in \cite{emss, AES_lat}, which produces a point in the more familiar space $X_2/\SO(2,\R)$.
\end{remark}
We consider $$X_{2,3}:=\{(\Lambda,w):\cov(\La)=1,w\in\mathbb{S}^2,w\perp\La\},$$ and note that $\s$ defined in \eqref{definition of shape} yields a map $$\s:X_{2,3}\to X_2/\SO(2,\R).$$

We define a right $\SL(3,\R)$ action on $X_{2,3}$ using the usual right matrix multiplication by \begin{equation} \label{action of SL on X_2,3}
   (\La,w). g:=\left(\frac{\La g}{\sqrt{\cov(\La g)}}, \frac{w (^tg^{-1})}{\|w (^tg^{-1})\|} \right),~g\in \SL(3,\R),~w\in \mathbb{S}^2, 
\end{equation}  where $\|\cdot\|$ is the usual Euclidean norm. 

It turns out that for each lattice subgroup $\Ga\leq \SL(3,\R)$ and each $x_0\in X_{2,3}$, the orbit $x_0\Ga$  is dense in $X_{2,3}$. This follows for example by \cite{Sargent2017DynamicsOT}, and a different proof of this fact is obtained by applying a duality argument. As we explain below $X_{2,3}$ is a homogeneous space identified with $H\backslash G$, and by Proposition 1.5 of \cite{Dani1980OrbitsOE}, it follows that all $H$-orbits in $G/\Ga$ are dense. 

Our goal in this paper will be to compute the limiting distribution of $\Ga$ orbits in $X_{2,3}$ with respect to growing Hilbert-Schmidt norm balls. Namely, let $\|g\|=\sqrt{\text{Trace}(^tgg)}=\sqrt{\sum_{ij}(g_{ij}^2)}$ be the Hilbert-Schmidt norm of $g\in\SL(3,\R)$, let $$\Ga_T:=\{\ga \in \Ga:\|\ga\|\leq T\},$$ and for $$x_0:=(\La_0,w_0)\in X_{2,3},$$ consider the probability measures $$\mu_{T,x_0}:=\frac{1}{\#\Ga_T}\sum_{\ga\in \Ga_T}\delta_{x_0.\gamma},~T>0.$$
Our main result is that the probability measures $\mu_{T,x_0}$ converge as $T\to\infty$ to a certain probability measure $\tilde\nu_{x_0}$ depending on ${x_0}\in X_{2,3}$ which we describe now. We observe that $X_{2,3}$ has a natural projection to $\mathbb{S}^2$ defined by $$\pi_{\perp}(\La,w):=w,$$ which endows $X_{2,3}$ with a fiber-bundle structure, where the fibers are isomorphic to $$X_2:=\SL(2,\Z)\backslash\SL(2,\R).$$ 
To define the measure $\tilde\nu_{x_0}$, we define a measure $\mu_{x_0,w}$ on each fiber $\pi_{\perp}^{-1}(w)$, and we integrate those measures by the unique rotation invariant probability measure $\mu_{\Sph^2}$ on $\Sph^2$. We note that the measures $\mu_{x_0,w}$ have a slightly surprising form; they are a combination of a $\SL(2,\R)$-invariant measure with a density involving the Hilbert-Schmidt norm of operators which we define now.

For an operator $T$ from a hyperplane $U\subset \R^3$ to another hyperplane $V\subset \R^3$, we define
\begin{equation}\label{defining hilbert-schmits for op. from hyp. to hyp.}
    \|T\|^2_{\text{HS}}:=\|Tu_1\|^2+\|Tu_2\|^2,
\end{equation}
where $\{u_1,u_2\}$ is an orthonormal basis of $U$, and where the norm on the right hand side is the usual Euclidean norm on $\R^3$. We note that this norm is independent of the choice of an orthonormal basis $\{u_1,u_2\}$. Moreover, this norm is bi-$\SO(3,\R)$ invariant in the following sense. If $\rho_1,\rho_2\in\SO(3,\R)$, then $\rho_2 \circ T \circ \rho_1:\rho_1^{-1}U\to \rho_2V$ satisfies $$\|\rho_2 \circ T \circ \rho_1\|_{\text{HS}}=\|T\|_{\text{HS}}.$$
For an ordered tuple of linearly independent vectors $B=(u_1,u_2)\in\R^3\times\R^3$ we define the linear map $T_B:\text{Span}_{\R}  \{e_1,e_2\} \to\text{Span}_{\R}\{u_1,u_2\}$, by sending $e_1\mapsto u_1$ and $e_2\mapsto u_2$. Now fix unimodular 2-lattice $\Lambda_0\subset\R^3$, and let $\mathscr{B}_0$ be an ordered tuple of linearly independent vectors forming a $\Z$-basis for $\La_0 $. We define for an arbitrary unimodular 2-lattice $\La\subset\R^3$, \begin{equation}\label{defining psi_La_0}
    \Psi_{\La_0} (\La):=\sum_{\Span_\Z{\mathscr{B}}=\La}\frac{1}{\|T_{\mathscr{B}}\circ T_{\mathscr{B}_0}^{-1}\|_{\text{HS}}^4}.
\end{equation}
We note that $\Psi_{\La_0}$ is independent of the choice of basis $\mathscr{B}_0$ (see \eqref{intrinsic interpretation of the sum}), and we observe that by bi-$\SO(3,\R)$ invariance of the Hilbert-Schmidt norm that the values of the function $\Psi_{\La_0}(\La)$ only depends on the shapes of $\La$ and $\La_0$.

By identifying $\pi_\perp^{-1}(e_3)$ with $\SL(2,\Z)\backslash \SL(2,\R)$, we obtain the $\SL(2,\Z)\backslash \SL(2,\R)$ invariant measure $\mu_{e_3}$ supported on $\pi_\perp^{-1}(e_3)$ scaled such that the measure $\mu_{x_0,e_3}$ defined by $$\mu_{x_0,e_3}(f):=\int_{\pi^{-1}_\perp(e_3)}f(\La,e_3)\Psi_{\La_0}(\La)d\mu_{e_3}(\La),~f\in C_c(X_{2,3}),$$ is a probability measure. Then, we obtain the measure supported on $\pi_\perp^{-1}(w), \text{ for } w\in \mathbb{S}^2$, by choosing $\rho_w\in \SO(3,\R)$ such that $w=e_3 \rho_w$, and defining,$$\mu_{x_0,w}:=(\rho_w)_*\mu_{x_0,e_3},$$ which is the push-forward of the right translation by $\rho_w$ via the right action of $\SO(3,\R)$ on $X_{2,3}$ defined in \eqref{action of SL on X_2,3}. Note that $\mu_{x_0,w}$ is independent of the choice of $\rho_w$. Finally, we define $\tilde\nu_{x_0}$ by $$\tilde\nu_{x_0}(f)=\int_{\mathbb{S}^2}\mu_{x_0,w}(f)d\mu_{\mathbb{S}^2}(w).$$   

The main goal of this paper is to prove the following:

\begin{theorem}\label{equidistribution result on  G mod H}
 Let $\Ga\leq\SL(3,\R)$ be a lattice and fix ${x_0}\in X_{2,3}$. Then, $\mu_{T,{x_0}}$ converges in the weak-* topology to $\tilde\nu_{x_0}$ as $T\to\infty$. In other words,
 for all $f\in C_c(X_{2,3})$, we have 
\begin{equation}
\lim_{T
\to\infty} \frac{1}{\#\Gamma_{T}}\sum_{\gamma\in \Ga_T}f(x_{0}\cdot \gamma)=\int_{X_{2,3}}f(x)d \tilde\nu_{x_{0}}(x).
\end{equation}

\end{theorem}
Now consider the map $$(\s\times\pi_\perp):X_{2,3}\to  X_2\slash \SO(2,\R)\times \Sph^2,$$ defined by $$(\s\times\pi_\perp)(\La,w):=(\s(\La,w),w).$$ We observe that that the push-forward of $\nu_{x_0}$ by $(\s\times\pi_\perp)$ is given by $$(\s\times\pi_\perp)_*(\nu_{x_0})=\s_*\mu_{x_0,e_3}\otimes\mu_{\mathbb{S}^2}.$$
\begin{corollary}
     Let $\Ga\leq\SL(3,\R)$ be a lattice and fix ${x_0}\in X_{2,3}$. Then, the probability measures on $X_2\slash \SO(2,\R)\times\Sph^2  $ given by $$(\s\times\pi_\perp)_*\mu_{T,{x_0}}=\frac{1}{\#\Ga_T}\sum_{\|\ga\|\leq T}\delta_{(\s\times\pi_\perp)({x_0}.\gamma)},~T>0,$$ 
    converge in the weak-* topology to the probability measure $(\s\times\pi_\perp)_*(\nu_{x_0})=\s_*\mu_{x_0,e_3}\otimes\mu_{\mathbb{S}^2}$ as $T\to\infty$.
\end{corollary}

\subsection{Connection to homogeneous dynamics - the duality principle}

To simplify notation, we denote $G:=\SL(3,\R)$. We note that the $G$-action on $X_{2,3}$ is transitive, and we observe that the stabilizer subgroup of the base point $(\Span_\Z\{(1,0,0),(0,1,0)\},(0,0,1))$ is  
\begin{equation}
H=\left \{\begin{bmatrix}
t^{-\frac{1}{2}}\ga & 0 \\
(x,y) &  t 
\end{bmatrix}:t> 0, \ga \in \SL(2,\Z), (x,y)\in \R^2 \right\}.
\end{equation} 
Then we obtain the identification $H \backslash G \cong X_{2,3}$.

\vspace{5mm}
This connects our problem to the study of distribution of orbits of closed subgroups in homogeneous spaces. In general, let $G$ be a connected Lie group, $\Ga$ be a lattice in $G$ and $H$ be a closed subgroup of $G$ (or more generally for two closed subgroups $H_1$ and $H_2$). The duality principle allows us to connect the equidistributional properties of the $\Ga$-orbits on $X$ to the equidistributional 
properties of the $H$-orbits in the dual action of $H$ on $G/\Ga$, see Section 1.7 of \cite{GN14duality} for an extensive exposition of the existing literature on this principle.  A general recipe for applying the duality principle was developed in \cite{GN14duality} and \cite{Gorodnik2004DistributionOL}.  Our approach in this paper uses a theorem of Gorodnik and Weiss (\cite{Gorodnik2004DistributionOL}), since it allows to prove equidistribution for every starting point. 

\vspace{5mm}
A novelty of our work is the application of the duality principle for a closed subgroup $H$ with infinitely many non-compact connected components. We note that the case when $H$ is connected algebraic was studied in great generality in \cite{GorodnikNevo2012,Gorodnik2004DistributionOL}, and the case where $H$ is discrete was studied in \cite{Oh05} and in greater generality in \cite{Gorodnik2004DistributionOL}.

To begin with, we shall study the invariant measure on $H$ and introduce an equidistribution theorem on $G/\Ga$.

\subsection{An equidistribution theorem on $G/\Ga$ with respect to the expanding skew balls in $H$}

Let us first describe the left invariant measure on $H$. Put
\begin{equation}
    \SL_{2,1}(\Z):=\begin{bmatrix}
\SL(2,\Z) & 0 \\
0 &  1 
\end{bmatrix}, A=\left \{\begin{bmatrix}
t^{-\frac{1}{2}}I_2 & 0 \\
0 & t 
\end{bmatrix}:t>0 \right\}, U:=\begin{bmatrix}
I_2 & 0 \\
\mathbb{\R}^2 & 1 
\end{bmatrix}
\end{equation}

Then the group $H$ has the decomposition:
$$H=U\rtimes(\SL_{2,1}(\Z)\times A)=U\rtimes(A \times \SL_{2,1}(\Z)).$$
Notice that any element in $H$ can be uniquely represented as $u\ga a$, where $u\in U$, $\ga \in \SL_{2,1}(\Z)$, $a \in A$ (note that $a$ commutes with $\ga$),

where the semidirect product is given by conjugation:
\begin{equation}
    (u_1,\ga_1,a_1)\cdot (u_2,\ga_2, a_2):= (u_1 \ga_1 a_1 u_2 (\ga_1 a_1)^{-1}, \ga_1 \ga_2, a_1 a_2)
\end{equation}


The left Haar measure $\mu$ on $H$ can be given by the following proposition:
\begin{proposition}\label{Haar measure on $H$}
For an integrable function $f=f(x)$ on $H$, write $x=u_v \ga a_t$, then under the parametrization

\begin{equation}
    \ga:=\begin{bmatrix}
\ga & 0 \\
0 &  1 
\end{bmatrix} \footnote{Here and henceforth we shall abuse the notation $\ga$, allowing it to represent both the $2\times 2$ and $3 \times 3$ matrices whenever its meaning is evident from the context.}, 
a_t:=\begin{bmatrix}
t^{-\frac{1}{2}}I_2 & 0 \\
0 & t 
\end{bmatrix}, 
u_v:=\begin{bmatrix}
I_2 & 0 \\
v & 1 
\end{bmatrix},   
\end{equation}
where $\ga\in \SL(2,\Z),t>0, v=(x,y)\in \R^2$, we have
\begin{align}
    \int_H f(x)d\mu(x)=& \int_{\SL_{2,1}(\Z)}\int_0^{\infty}\int_{\R^2} f(u_va_t\ga) dv \frac{1}{t^4} dt d\ga\\
    =&\sum_{\ga\in \SL_{2,1}(\Z)}\int_0^{\infty}\int_{\R^2} f(u_v a_t\ga) dv \frac{1}{t^4} dt, 
\end{align}
where the measure $dv \frac{1}{t^4} dt d\ga=dx dy \frac{1}{t^4} dt d\ga$ is left $H$-invariant.
\end{proposition}

\begin{proof}
 To check the left invariance, fix $\ga_0 \in \SL_{2,1}(\Z)$, $a_{t_0}:=\begin{bmatrix}
t_0^{-\frac{1}{2}}I_2 & 0 \\
0 & t_0 
\end{bmatrix}, t_0\ne 0$ and $u_{v_0}:=\begin{bmatrix}
I_2 & 0 \\
v_0 & 1 
\end{bmatrix}$, $v_0:=(x_0,y_0)\in \R^2$, then 
\begin{align*}
      u_{v_0} a_{t_0} \ga_0 \cdot u_v a_t \ga
      =& u_{v_0} [a_{t_0} \ga_0 u_v (a_{t_0} \ga_0)^{-1}] a_{t_0} \ga_0 a_t \ga \\
      =& u_{v_0} u_{t_0^{\frac{3}{2}} v\ga_0^{-1}}a_{t_0} a_t \ga_0 \ga \\
      =& u_{v_0+t_0^{\frac{3}{2}} v\ga_0^{-1}}a_{t_0 t} \ga_0 \ga
\end{align*}
It follows that
\begin{align*}
 & \int_{\SL_{2,1}(\Z)}\int_0^{\infty}\int_{\R^2} L_{u_{v_0} a_{t_0} \ga_0}[f(ua\ga) ] dx dy \frac{1}{t^4}dt d\ga \\
 =&  \int_{\SL_{2,1}(\Z)}\int_0^{\infty}\int_{\R^2} f(u_{v_0} a_{t_0} \ga_0ua\ga)  dx dy \frac{1}{t^4}dt d\ga \\
 =&  \int_{\SL_{2,1}(\Z)}\int_0^{\infty}\int_{\R^2} f(u_{v_0+t_0^{\frac{3}{2}} v\ga_0^{-1}}a_{t_0 t} \ga_0 \ga) dx dy \frac{1}{t^4}dt  d\ga
\end{align*}
By the change of variable
\begin{equation*}
    v:=(x,y)\mapsto v':=(x',y')=v_0+t_0^{\frac{3}{2}}v\ga_0^{-1}, t\mapsto t':=t_0t, \ga \mapsto \ga':=\ga_0 \ga,
\end{equation*}
and noticing that $\ga_0$ (with determinant $1$) does not contribute to the Jacobian, the last line of the equation above becomes
\begin{equation*}
    \int_{\SL_{2,1}(\Z)}\int_0^{\infty}\int_{\R^2} f(u_{v'}a_{t'}\ga') (t_0^{-\frac{3}{2}})^2 dx' dy' \frac{t_0^4}{t'^4}d(\frac{t'}{t_0})  d\ga'=    \int_{\SL_{2,1}(\Z)}\int_0^{\infty}\int_{\R^2} f(u_{v'}a_{t'}\ga') dx' dy' t'^4 dt'  d\ga'
\end{equation*}
This verifies the invariance of measure. Finally, by the uniqueness of Haar measure, the measure $d\ga$ on $\SL_{2,1}(\Z)$ is the counting measure and therefore the out-most integral reduces to the sum. 
\end{proof}

Let $\|\cdot\|$ denote the Hilbert-Schmidt norm on matrices. Namely $\|A\|:=\sqrt{\text{Trace}(A^TA)}$, or the square root of the sum of squares of all entries of the matrix $A$.

For any subgroup $L$ of $G=\SL(3,\R)$, let
\begin{equation}\label{meaning of the subscript T}
    L_T:=\{g\in L: \|g\|<T\}.
\end{equation}

Following \cite{Gorodnik2004DistributionOL}, for $g_1,g_2,g \in G$ and $T>0$, we define the so-called ``skewed balls" as follows:
\begin{equation}
    H_T[g_1,g_2]:=\{h\in H: \|g_1hg_2\|<T \}.
\end{equation}
Let 
\begin{equation}
    V_{\ga,T}[g_1,g_2]:=\{h\in U A \ga: \|g_1hg_2\|<T \},
\end{equation}
then it follows that 
\begin{equation}\label{decomposition of H_T into V_T}
    H_T[g_1,g_2]=\bigsqcup_{\ga \in \SL_{2,1}(\Z)}V_{\ga,T}[g_1,g_2].
\end{equation}

\vspace{5mm}

For $F\in C_c(G/\Ga)$ and $g_1,g_2\in G$, we consider the measure defined by the integral 
\begin{equation}\label{definition of measure funtional}
    \phi_{\mu}(F,T)=\phi_{\mu}(F,T,g_1,g_2):=\frac{1}{\mu \left(H_T[g_1,g_2] \right)}\int_{H_T[g_1,g_2]}F(h^{-1}g_1\Gamma)d\mu(h).
\end{equation}

In view of Theorem 2.3 in \cite{Gorodnik2004DistributionOL}, the following equidistribution theorem plays the central role in the proof of Theorem \ref{equidistribution result on  G mod H}:

\begin{theorem}\label{The $G$-invariance of limiting measure}
Let $H$ and $H_T$ be defined as above, and let $\mu=\mu_H$ denote the left Haar measure on $H$, and $\mu_X$ denote the normalized $G$-invariant probability measure on $X=G/\Ga$. For any $f\in C_c(X)$, let $\phi_{\mu}(F,T)$ be defined as in \ref{definition of measure funtional} with $g_1, g_2\in G$ arbitrary. Then
\begin{equation}
    \lim_{T\to \infty} \phi_{\mu}(F,T)= \int_{X}F(x)d\mu_X(x).
\end{equation}
In other words, the limiting measure is the $G$-invariant probably measure on $X$. 
\end{theorem}

\subsection{Structure of the paper and further work}

\begin{itemize}
    \item The outline of our proof of Theorem \ref{The $G$-invariance of limiting measure} is as follows. We will first give an estimate for the volume of $H_T[g_1,g_2]$ (Section 2). This estimate will help us to show the unipotent invariance of limiting measure (Section 3), which in turn opens the gate to the application of Ratner theory. In particular, a dichotomy theorem (Theorem \ref{Shah dichotomy theorem}) due to Shah plays the central role here. To exclude the first outcome of the dichotomy theorem, we use a lemma (Lemma \ref{very important inequaity lemma of nimish-gorodnik}) on the expansion of the norms of vectors under the action of unipotent subgroups. This allows us to prove the non-escape of mass on the space of one-point compactification of $G/\Ga$ and the $G$-invariance will follow from Ratner's measure classification theorem (\cite{Ratner91a}, \cite{Mozes1995OnTS}).
    \item In Section 5 we prove Theorem \ref{equidistribution result on  G mod H} as a consequence of the duality principle stated in Theorem 2.2 to Corollary 2.4 in \cite{Gorodnik2004DistributionOL}. 
\end{itemize}

\vspace{5mm}
We remark that despite that the theorem is stated from the case of rank-two lattices in $\R^3$. Our work in progress is to generalize the results to higher ranks and dimensions. The main difficulty lies in the estimtate of the volumes of balls as well as the expansion of norms under unipotent actions. 

\section*{Acknowledgements}
The authors would like to thank Nimish Shah for helpful discussions and generous sharing of his ideas. We would also like to thank Uri Shapira and Barak Weiss for their comments on the manuscript.

\section{Estimate of the Haar measure growth of skewed balls in $H$}

For $g_1, g_2 \in G$, by using Iwasawa decomposition (for block-lower-triangular matrix), we may assume 
\begin{equation}\label{matrix representation of g1 and g2}
    g_1:=k_1\begin{bmatrix} G_1 & 0 \\G_3 & G_4 \end{bmatrix}, g_2:=\begin{bmatrix} H_1 & 0 \\H_3 & H_4 \end{bmatrix} k_2
\end{equation}
where $k_1,k_2 \in \SO(3,\R)$, $G_1,H_1\in \R^{2\times 2}$, $G_3, H_3\in \R^{2\times 1}$ and $G_4, H_4\in \R$.

The main goal of this section is to prove the following:
\begin{proposition}\label{computation for H and V} Recall 
    $H_T[g_1,g_2]:=\{h\in H: \|g_1hg_2\|<T \}$ and $V_{\ga,T}[g_1,g_2]:=\{h\in U A \ga: \|g_1hg_2\|<T \}$. Then
    \begin{equation}
    V_{\ga,T}[g_1,g_2]=\frac{\pi}{G_4^2 |\det(H_1)|}\frac{T^6}{6\|G_1\ga H_1\|^4}+O_{g_1,g_2}(T^4),
and
   \end{equation}
   \begin{equation}
    H_T[g_1,g_2]=\frac{\pi}{G_4^2 |\det(H_1)|}\frac{T^6}{6}\sum_{\ga\in \SL(2,\Z)}\frac{1}{\|G_1\ga H_1\|^4}+O_{g_1,g_2}(T^4).  
   \end{equation}
Moreover, if $g_1$ and $g_2$ lie in a bounded subset of $G$, then the implied constant in $O_{g_1,g_2}$ is bounded.

In particular, for $g_1=g_2=e$:

\begin{equation}
    \mu(H_T)=\frac{\pi}{6}\sum_{\ga\in \SL(2,\Z)}\frac{1}{\|\ga\|^4} T^6+O(T^4).
\end{equation}
\end{proposition}

\begin{proof}
Recall that our $H_T[g_1,g_2]$ is the disjoint union of 
    $V_{\ga,T}[g_1,g_2]:=\{h\in U A \ga: \|g_1hg_2\|<T \}$. We will first evaluate $V_{\ga,T}[g_1,g_2]$ and then take the sum over $\SL(2,\Z)$. Since $k_1,k_2$  preserve the Hilbert-Schmidt norm, we may assume $k_1$ and $k_2$ are identity matrices.

Every matrix in $H$ can be written as
\begin{equation*}
\begin{bmatrix}
I_2 & 0 \\
v & 1 
\end{bmatrix}
\begin{bmatrix}
t^{-\frac{1}{2}}I_2 & 0 \\
0 & t 
\end{bmatrix}
\begin{bmatrix}
\ga & 0 \\
0 &  1 
\end{bmatrix}    
=\begin{bmatrix}
t^{-\frac{1}{2}}\ga & 0 \\
t^{-\frac{1}{2}}v\ga &  t 
\end{bmatrix}, 
\end{equation*}
where $t\ne 0, \ga \in \SL(2,\Z), v=(x,y)\in \R^2$.
Note
\begin{align*}
    g_1hg_2=&\begin{bmatrix} G_1 & 0 \\G_3 & G_4 \end{bmatrix}
    \begin{bmatrix}
t^{-\frac{1}{2}}\ga & 0 \\
t^{-\frac{1}{2}}v\ga &  t 
\end{bmatrix}
    \begin{bmatrix} H_1 & 0 \\H_3 & H_4 \end{bmatrix}\\
    =&\begin{bmatrix}
t^{-\frac{1}{2}}G_1\ga & 0 \\
t^{-\frac{1}{2}}G_3\ga+t^{-\frac{1}{2}}G_4v\ga &  G_4t 
\end{bmatrix}\begin{bmatrix} H_1 & 0 \\H_3 & H_4 \end{bmatrix}\\
=&\begin{bmatrix}
t^{-\frac{1}{2}}G_1\ga H_1 & 0 \\
t^{-\frac{1}{2}}G_3\ga H_1+t^{-\frac{1}{2}}G_4v\ga H_1 + G_4 H_3t &  G_4H_4t 
\end{bmatrix}
\end{align*}
Upon taking the sum of squares and rearranging terms, we conclude that $\|g_1hg_2\| \le T$ is equivalent to
\begin{equation}\label{equation defining the disk of integration}
    \|G_3\ga H_1+G_4 v\ga H_1+ G_4 H_3 t^{\frac{3}{2}}\|^2 \le -G_4^2 H_4^2 t^3+tT^2-\|G_1\ga H_1\|^2
\end{equation}

In view of \eqref{equation defining the disk of integration}, let 
\begin{equation}\label{the notation D for the disk}
    D=D_{\ga, t}:=\{v\in \R^2: \|G_3\ga H_1+G_4 v\ga H_1 + G_4 H_3 t^{\frac{3}{2}}\|^2 \le -G_4^2 H_4^2 t^3+tT^2-\|G_1\ga H_1\|^2\}.
\end{equation}

Note
\begin{equation}
    \mu \left(V_{\ga,T}[g_1,g_2]\right):=\int_0^{\infty}\int_{\R^2} \mathbf{1}_{V_{\ga,T}[g_1,g_2]}(u_va_t\ga) dv \frac{1}{t^4}dt=\int_0^{\infty}m(D_{\ga,t})\frac{1}{t^4}dt,
\end{equation}
where $m(D_{\ga,t})$ is the Lebesgue measure of $D_{\ga,t}$ with the understanding that it is zero for $t$ with $-G_4^2 H_4^2 t^3+tT^2-\|G_1\ga H_1\|^2<0$.

We  observe that $D_{\ga,t}$ is the interior of an ellipse on $\R^2$, whose area is 
\begin{equation}
    \pi \frac{-G_4^2 H_4^2 t^3+tT^2-\|G_1\ga H_1\|^2}{G_4^2 |\det(H_1)|}.
\end{equation}


The maximal value of $-G_4^2 H_4^2 t^3+tT^2$ for $t>0$ is $M_{T}:=\frac{2T^3}{3\sqrt{3}G_4 H_4}$, attained at $t_0:=\frac{T}{\sqrt 3 G_4H_4}$. It is easy to see that $-G_4^2 H_4^2 t^3+tT^2-\|G_1\ga H_1\|^2$ has two positive roots which we denote by $a=a_{\ga,T}$ and $b=b_{\ga,T}$ with $0<a<b<\frac{T}{G_4H_4}$, whenever $\|G_1\ga H_1\|^2<M_T$. Moreover, $-G_4^2 H_4^2 t^3+tT^2-\|G_1\ga H_1\|^2$ is positive on $(a,b)$. Observe that we only have to consider those $\ga$ with $\|G_1\ga H_1\|^2<M_T$ and $0<a<t_0<b<\sqrt{3}t_0$.

Since $-G_4^2 H_4^2 a^3+aT^2-\|G_1\ga H_1\|^2=0$,
\begin{equation}\label{range for a}
    a=\frac{\|G_1\ga H_1\|^2}{T^2-G_4^2H_4^2 a^2}\in
    \left(\frac{\|G_1\ga H_1\|^2}{T^2}, \frac{3\|G_1\ga H_1\|^2}{2T^2}\right).
\end{equation}

In particular,
\begin{equation}
        a=\Theta \left( \frac{\|G_1\ga H_1\|^2}{T^2} \right) \label{estimate for a in all range}
\end{equation}

Meanwhile by Taylor expansion,
\begin{align}
    a=&\frac{\|G_1\ga H_1\|^2}{T^2}\frac{1}{1-\frac{G_4^2H_4^2 a^2}{T^2}} \nonumber\\ 
    =&\frac{\|G_1\ga H_1\|^2}{T^2}\left(1+\frac{G_4^2H_4^2 a^2}{T^2}+\cdots \right) \label{taylor expansion for a}
\end{align}

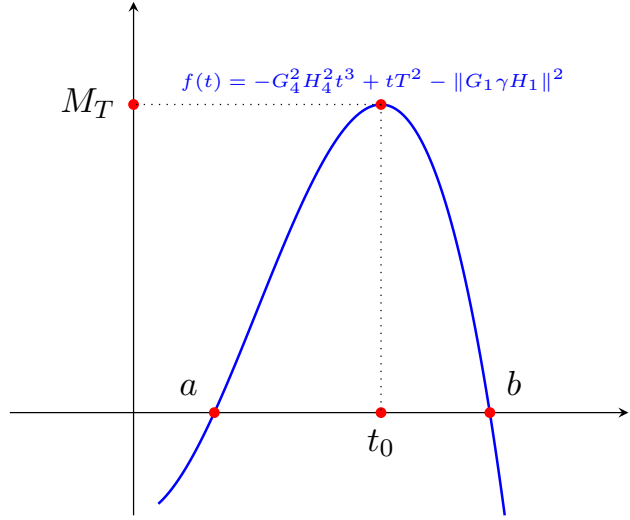
\begin{figure}[hbt!] 
    \centering
    \pgfplotsset{ticks=none}
\begin{tikzpicture}[scale=1.2]
    \begin{axis}[
        axis lines=center,
        xmax = 4,
        xmin = -1,
        ymax = 4,
        ]
        \addplot [domain=0.2:3,samples=250, thick, blue] {-x^3+3*x^2-1}
        node [font=\tiny, pos=0.5, above]{$f(t)=-G_4^2 H_4^2 t^3+tT^2-\|G_1\ga H_1\|^2$};
        \node[red,label={[yshift=0.1ex]135:{$a$}},circle,fill,inner sep=1.2pt] at (axis cs:0.6527,0) {};
        \node[red,label={[yshift=0.1ex]45:{$b$}},circle,fill,inner sep=1.2pt] at (axis cs:2.88,0) {};
        \node[red,label={[yshift=0.1ex]45:{}},circle,fill,inner sep=1.2pt] at (axis cs:2,3) {};
        \node[label={[yshift=0.1ex]180:{$M_T$}},circle,fill=red,inner sep=1.2pt] at (axis cs:0,3) {};
        \node[label={[yshift=0.1ex]270:{$t_0$}},circle,fill=red,inner sep=1.2pt] at (axis cs:2,0) {};
        \addplot+[
      black,thin,dotted,
      mark=none,
      const plot,
      empty line=jump,
    ]
    coordinates {
        (2,3)
        (2,0)
    };
      \addplot+[
      black,thin,dotted,
      mark=none,
      const plot,
      empty line=jump,
    ]
    coordinates {
        (2,3)
        (0,3)
    };
    \end{axis}
    
    \end{tikzpicture}
    \label{fig:finding the roots of cubic equation}
    \caption{Estimating the roots of the polynomial.}
\end{figure}

Moreover, for the $\ga$ with smaller norms, it follows from \eqref{taylor expansion for a} that
\begin{equation}\label{estimate for a in smaller range}
    a=\frac{\|G_1\ga H_1\|^2}{T^2}\left(1+O_{g_1,g_2}\left(\frac{1}{T^2} \right)\right), \text{ whenever }\|G_1\ga H_1\|\le T.
\end{equation}

\vspace{5mm}
Now we compute the measure of $H_T[g_1,g_2]$. By Proposition \ref{Haar measure on $H$} with $f=\mathbf{1}_{H_T[g_1,g_2]}$,
\begin{align}
    \mu \left(H_T[g_1,g_2] \right)
    =& \sum_{\substack{\ga\in \SL(2,\Z),\\ \|G_1\ga H_1\|^2<M_T}} \mu \left(V_{\ga,T}[g_1,g_2]\right)\\
    = &\sum_{\substack{\ga\in \SL(2,\Z),\\ \|G_1\ga H_1\|^2<M_T}} \frac{\pi }{G_4^2 |\det(H_1)|}\int_a^b \frac{-G_4^2 H_4^2 t^3+tT^2-\|G_1\ga H_1\|^2}{t^4}dt \nonumber\\
    =&\sum_{\substack{\ga\in \SL(2,\Z),\\ \|G_1\ga H_1\|^2<M_T}} \frac{\pi }{G_4^2 |\det(H_1)|}\int_a^b -\frac{G_4^2 H_4^2}{t}+\frac{T^2}{t^3}-\frac{\|G_1\ga H_1\|^2}{t^4}dt 
\end{align}
But
\begin{align*}
    &\int_a^b -\frac{G_4^2 H_4^2}{t}+\frac{T^2}{t^3}-\frac{\|G_1\ga H_1\|^2}{t^4}dt\\
    =& \left(-G_4^2 H_4^2 \log(t)-\frac{T^2}{2t^2}+\frac{\|G_1\ga H_1\|^2}{3t^3}\right)\Bigg|_{t=a}^{t=b}\\
    =& \left(-G_4^2 H_4^2 \log(b)-\frac{T^2}{2b^2}+\frac{\|G_1\ga H_1\|^2}{3b^3}\right)-\left(-G_4^2 H_4^2 \log(a)-\frac{T^2}{2a^2}+\frac{\|G_1\ga H_1\|^2}{3a^3}\right)
\end{align*}
Since $0<a<t_0<b<\sqrt{3}t_0:=\frac{T}{G_4H_4}$, $\|G_1\ga H_1\|^2<M_T=\frac{2T^3}{3\sqrt{3}G_4 H_4}$ and by \eqref{range for a}, the expression above is asymptotically 
\begin{equation}\label{before plugging in a}
    \frac{T^2}{2a^2}-\frac{\|G_1\ga H_1\|^2}{3a^3}+O_{g_1,g_2}(\log(T)).
\end{equation}

Now by \eqref{estimate for a in all range}, the equation \eqref{before plugging in a} has the estimate 
\begin{equation}\label{estimate for V ga T}
   \Theta \left(\frac{T^6}{\|G_1\ga H_1\|^4} \right), 
\end{equation}
for all $\ga$.

In view of \eqref{estimate for a in smaller range}, if furthermore $\|G_1\ga H_1\|\le T$, then \eqref{before plugging in a} becomes
\begin{align}
    =&\frac{T^6}{2\|G_1\ga H_1\|^4}\left(1+O_{g_1,g_2}\left(\frac{1}{T^2} \right)\right)^{-2}-\frac{T^6}{3\|G_1\ga H_1\|^4}\left(1+O_{g_1,g_2}\left(\frac{1}{T^2} \right)\right)^{-3}\nonumber\\
    =&\frac{T^6}{2\|G_1\ga H_1\|^4}(1-O_{g_1,g_2}(\frac{1}{T^2})+\cdots)^2-\frac{T^6}{3\|G_1\ga H_1\|^4}(1-O_{g_1,g_2}(\frac{1}{T^2})+\cdots)^3 \nonumber\\
    =&\frac{T^6}{6\|G_1\ga H_1\|^4}-\frac{1}{\|G_1\ga H_1\|^4}O_{g_1,g_2}\left(T^4\right)
\end{align}
It follows that 
\begin{align}
    \mu \left(H_T[g_1,g_2] \right)
    = & \frac{\pi}{G_4^2 |\det(H_1)|}\sum_{\substack{\ga\in \SL(2,\Z), \nonumber\\ \|G_1\ga H_1\|\le T}}\left(\frac{T^6}{6\|G_1\ga H_1\|^4}+\frac{1}{\|G_1\ga H_1\|^4}O\left(T^4\right)\right) \nonumber\\
    +&\frac{\pi }{G_4^2 |\det(H_1)|}\sum_{\substack{\ga\in \SL(2,\Z),\\ T<\|G_1\ga H_1\|<\sqrt{M_T}}}   O\left(\frac{T^6}{\|G_1\ga H_1\|^4} \right) \nonumber \\
    =&  \frac{\pi}{G_4^2 |\det(H_1)|}\frac{T^6}{6}\sum_{\ga\in \SL(2,\Z)}\frac{1}{\|G_1\ga H_1\|^4}+O_{g_1,g_2}(T^4), \label{measure of H_T}
\end{align}

where the last line of the equation is the consequence of the lemma below.
\end{proof}

\begin{lemma}
For any invertible $2\times 2$ matrices $A$ and $B$ and positive numbers $\alpha>2$ and $K$, we have
\begin{equation}\label{head estimate for the sum over SL2}
    \sum_{\substack{\ga\in \SL(2,\Z),\\ \|A\ga B\|\le K}}\frac{1}{\|A\ga B\|^{\alpha}}
\end{equation}
is convergent as $K\to \infty$ and we have the tail estimate:
\begin{equation}\label{tail estimate for the sum over SL2}
    \sum_{\substack{\ga\in \SL(2,\Z),\\ \|A\ga B\|> K}}\frac{1}{\|A\ga B\|^{\alpha}}=O_{A,B}(K^{2-\alpha})
\end{equation}
\end{lemma}
\begin{proof}
Let $0=\tau_0<\tau_1<\tau_2<\tau_3 \cdots$ be an enumeration of $\|A\ga B\|$ with $\ga \in \SL(2,\Z)$. For $i=1,2,3\cdots$ and $\tau>0$, let 
\begin{equation}
    N(\tau):=\#\{\ga\in \SL(2,\Z):\|A\ga B \|\le \tau\}
\end{equation}
and for $K>0$ $M(K):=\max\{i:N(\tau_i)\le K\}$. Observe that $M(K)\gtrsim
 K$ and  $N(\tau_i)-N(\tau_{i-1})=\#\{\ga\in \SL(2,\Z):\|A\ga B \|=\tau_i\}$.


It follows from \cite{GorodnikNevo2012} (see Appendix B for computation details) that for $A=B=I_2$
\begin{equation}
    N(\tau)=c\tau^2+O_{\eta}(\tau^{\frac{5}{3}+\eta}).
\end{equation}

Since for operator norms we have the inequality $$\|A^{-1}\|^{-1}_{\text{op}}\|\ga \|_{\text{op}}\|B^{-1}\|^{-1}_{\text{op}} \le \|A\ga B\|_{\text{op}}\le \|A\|_{\text{op}}\|\ga\|_{\text{op}}\|B\|_{\text{op}},$$
and all norms on finite-dimensional spaces are equivalent, we have for general $A$ and $B$, $N(\tau)\asymp O_{A,B}(\tau^2)$. For $K>0$,

\begin{align}
     \sum_{\substack{\ga\in \SL(2,\Z),\\ \|A\ga B\|\ge K}}\frac{1}{\|A\ga B\|^{\alpha}}
     =&\sum_{i=M(K)}^{\infty}\frac{N(\tau_i)-N(\tau_{i-1})}{\tau_i^\alpha}\\
     =&\sum_{i=M(K)}^{\infty}\frac{N(\tau_i)}{\tau_i^\alpha}-\sum_{i=1}^{\infty}\frac{N(\tau_{i})}{\tau_{i+1}^\alpha}\tag{sum is convergence when $\alpha>2$}\\
     =&\sum_{i=M(K)}^{\infty}\left(\frac{N(\tau_i)}{\tau_i^\alpha}-\frac{N(\tau_{i})}{\tau_{i+1}^\alpha}\right)\\
     = &\sum_{i=M(K)}^{\infty}\int_{\tau_i}^{\tau_{i+1}}\frac{\alpha N(\tau)}{\tau^{\alpha+1}}d\tau\\ 
      =&\int_{\tau_{M(K)}}^{\infty}\frac{\alpha N(\tau)}{\tau^{\alpha+1}}d\tau\\ 
     \lesssim_{A,B} &\int_{\tau_{M(K)}}^{\infty}\frac{\alpha c }{\tau^{\alpha-1}}d\tau+\int_{\tau_{M(K)}}^{\infty}\frac{\alpha O_{\eta}(\tau^{\frac{5}{3}+\eta}) }{\tau^{\alpha+1}}d\tau\\ 
     \asymp & \tau_{M(K)}^{2-\alpha} \le K^{2-\alpha}
    \end{align}

This gives \eqref{head estimate for the sum over SL2}. The proof for \eqref{tail estimate for the sum over SL2} is similar.
\end{proof}

\vspace{5mm}
The following properties follow immediately from \eqref{measure of H_T}:
\begin{proposition}[Uniform volume growth for skewed balls in $H$, property D1 in \cite{Gorodnik2004DistributionOL}]
For any bounded subset $B\subset G$ and any $\e>0$, there are $T_0$ and $\de>0$ such that for all $T>T_0$ and all $g_1,g_2 \in B$ we have:
\begin{equation}
    \mu \left(H_{(1+\de)T}[g_1,g_2] \right)\le (1+\e)\mu \left(H_T[g_1,g_2] \right).
\end{equation}
\end{proposition}

\begin{proposition}[Limit volume ratios, property D2 in \cite{Gorodnik2004DistributionOL}]
For any $g_1,g_2 \in G$. the limit 
\begin{equation}
    \alpha(g_1,g_2):=\lim_{T\to \infty}\frac{\mu \left(H_T[g_1^{-1},g_2] \right)}{\mu \left(H_T \right)},
\end{equation}
exists and is positive and finite. 
\end{proposition}
\begin{proof}
It follows from Proposition \ref{computation for H and V} that 
\begin{equation}\label{definition and computation of alpha}
    \alpha(g_1^{-1},g_2):=\lim_{T\to \infty}\frac{\mu \left(H_T[g_1,g_2] \right)}{\mu \left(H_T \right)}=\frac{1}{G_4^2 |\det(H_1)|}\frac{\sum_{\ga\in \SL(2,\Z)}\frac{1}{\|G_1\ga H_1\|^4}}{\sum_{\ga\in \SL(2,\Z)}\frac{1}{\|\ga \|^4}}.
\end{equation}
\end{proof}

\section{The $U$-invariance of limiting measure}
Let $F\in C_c(G/\Ga)$, for $g_1,g_2\in G$ with matrix representation as in \eqref{matrix representation of g1 and g2}, recall the measure defined via the funtional 
\begin{equation}
    \phi_{\mu}(F,T)=\phi_{\mu}(F,T,g_1,g_2):=\frac{1}{\mu \left(H_T[g_1,g_2] \right)}\int_{H_T[g_1,g_2]}F(h^{-1}g_1\Gamma)d\mu(h)
\end{equation}

By Banach-Alaoglu theorem, for any fixed, $\phi_{\mu}(\cdot,T), T>0$ form a sequencially compact set in weak-star topology. By passing to a subsequence if necessary and with an abuse of notations, henceforth we assume that $\lim_{T\to \infty}\phi_{\mu}(F,T)$ exists. 

The main goal of this section is to prove that the limiting measure is $U$-invariant:
\begin{theorem}\label{unipotent invariance of haar measure}
$\lim_{T\to \infty}\phi_{\mu}(F,T)$ is left-invariant under $U$, where $U:=\begin{bmatrix}
I_2 & 0 \\
\mathbb{\R}^2 & 1 
\end{bmatrix}$. Specifically, for any $u_0\in U$, define $L_{u_0}(F)(x):=F(u_0^{-1}x)$, then 
\begin{equation}
    \lim_{T\to \infty}\left(\phi_{\mu}(F,T)-\phi_{\mu}(L_{u_0} F,T)\right)=0.
\end{equation}
\end{theorem}

In view of \eqref{decomposition of H_T into V_T}, we can decompose $\phi_{\mu}(F,T)$ as the following convex linear combination:
\begin{equation}
    \phi_{\mu}(F,T)= \sum_{\substack{\ga\in \SL(2,\Z),\\ \|G_1\ga H_1\|^2<M_T}}\frac{\mu \left(V_{\ga,T}[g_1,g_2] \right)}{\mu \left(H_T[g_1,g_2] \right)}\cdot \frac{1}{\mu \left(V_{\ga,T}[g_1,g_2] \right)}\int_{V_{\ga,T}[g_1,g_2]} F(h^{-1}g_1\Gamma) d\mu(h).
\end{equation}

\vspace{5mm}
We will first prove that individual terms are $U$-invariant. The following elementary lemma will be needed: 

\begin{lemma}\label{lemma on symmetric difference}
Let $f:\R^{d}\to \R$ be a bounded continuous function. Let $E\subset \R^{d}$ be a ellipsoid with surface area $S$, then for $y\in \R^{d}$, with $\|y\|$ less than the shortest radius of $E$, we have the estimate:
\begin{equation}
    \left|\int_{E}[f(v)-f(y+v)]dv\right| \le C_f \|y\|S,
\end{equation}
where $C_f$ is a constant depending on $f$.
\end{lemma}
\begin{proof}
Indeed, 
\begin{align*}
    \left|\int_{B(x,r)}[f(v)-f(y+v)]dv\right|
  = &\left|\int_{B(x,r)}f(v)dv - \int_{B(x+y,r)}f(v)dv\right|\\
  \le & \int_{B(x,r)\triangle
 B(x+y,r)}|f(v)|dv\\
  \le & C_f \|y\|S, \tag{$f$ is bounded}
\end{align*}
where the last line follows from the Theorem 1 of \cite{symmetricdifferencesch2010}.
\end{proof}

\begin{lemma}For any $\ga\in \SL(2,\Z)$ with $\|G_1\ga H_1\|^2<M_T$, let
\begin{equation}
    \phi_{\ga,\mu}(F,T):=\frac{1}{\mu \left(V_{\ga,T}[g_1,g_2] \right)}\int_{V_{\ga,T}[g_1,g_2]} F(h^{-1}g_1\Gamma) d\mu(h).
\end{equation}
For any $u_0\in U$, define $L_{u_0}(F)(x):=F(u_0^{-1}x)$, then 
\begin{equation}
    \lim_{T\to \infty}\left(\phi_{\ga,\mu}(F,T)-\phi_{\ga,\mu}(L_{u_0} F,T)\right)=0.
\end{equation}
\end{lemma}

\begin{proof}Assume that $F\in C_c(G/\Ga)$ is bounded by $M_F$.
Recall from Proposition\eqref{computation for H and V} that 
\begin{equation}\label{estimate for the measure of V_T}
   \mu \left(V_{\ga,T}[g_1,g_2] \right)=\Theta \left(\frac{T^6}{\|G_1\ga H_1\|^4} \right), 
\end{equation}
Fix $u_0:=u_{v_0}=U:=\begin{bmatrix}
I_2 & 0 \\
v_0 & 1 
\end{bmatrix}\in U$. The integral
\begin{align*}
    &\left| \int_{V_{\ga,T}[g_1,g_2]} F(h^{-1}g_1\Gamma) d\mu(h)-\int_{V_{\ga,T}[g_1,g_2]} F(u_0^{-1}h^{-1}g_1\Gamma) d\mu(h) \right |\\
    =&\left|\int_{V_{\ga,T}[g_1,g_2]} \left(F(h^{-1}g_1\Gamma) - F(u_0^{-1}h^{-1}g_1\Gamma) \right) d\mu(h)\right |\\
    =&\left|\int_{a}^b \int_{D} \left(F(\ga^{-1} a_t^{-1} u_v^{-1} g_1\Gamma) - F(u_{v_0}^{-1}\ga^{-1} a_t^{-1} u_v^{-1}g_1\Gamma) \right) dv \frac{1}{t^4}dt\right |\\
    =&\left|\int_{a}^b \int_{D} \left(F(\ga^{-1} a_t^{-1} u_v^{-1} g_1\Gamma) - F(\ga^{-1} (\ga u_{v_0}^{-1}\ga^{-1}) a_t^{-1} u_v^{-1}g_1\Gamma) \right) dv \frac{1}{t^4}dt\right |\\
    =&\left|\int_{a}^b \int_{D} \left(F(\ga^{-1} a_t^{-1} u_v^{-1} g_1\Gamma) - F(\ga^{-1}a_t^{-1} [a_t\ga u_{v_0}^{-1}\ga^{-1}a_t^{-1}]  u_v^{-1}g_1\Gamma) \right) dv \frac{1}{t^4}dt\right |\\
    =&\left|\int_{a}^b \int_{D} \left(F(\ga^{-1} a_t^{-1} u_{-v} g_1\Gamma) - F(\ga^{-1}a_t^{-1} u_{-t^{\frac{3}{2}}v_0\ga^{-1}-v}  g_1\Gamma) \right) dv \frac{1}{t^4}dt\right |\\
    =&\left| \left(\int_a^c+\int_c^b \right) \int_{D} \left(F(\ga^{-1} a_t^{-1} u_v^{-1} g_1\Gamma) - F(\ga^{-1}a_t^{-1} u_{-t^{\frac{3}{2}}v_0\ga^{-1}-v}  g_1\Gamma) \right) dv \frac{1}{t^4}dt\right |,\tag{note that $u_v^{-1}=u_{-v}$}
\end{align*}
where, as before, $a=a_{\ga,T}$ and $b=b_{\ga,T}$ are two positive roots of the cubic equation $-G_4^2 H_4^2 t^3+tT^2-\|G_1\ga H_1\|^2=0$ and satisfy $0<a<t_0<b<\sqrt{3}t_0$ with $t_0:=\frac{T}{\sqrt 3 G_4H_4}$. Here 
\begin{equation}\label{the truncation c}
    c=c_{\ga,T}=\frac{\|G_1\ga H_1\|^{1.5}}{T^{1.5}}
\end{equation}
is a parameter for the separation of integral. The domain $D$ is as defined in \eqref{the notation D for the disk}, whose area is $\pi \frac{-G_4^2 H_4^2 t^3+tT^2-\|G_1\ga H_1\|^2}{G_4^2 |\det(H_1)|}$.

The second integral $\int_c^b$ has the estimate
\begin{align*}
    &\left| \int_c^b  \int_{D} \left(F(\ga^{-1} a_t^{-1} u_{-v} g_1\Gamma) - F(\ga^{-1}a_t^{-1} u_{-t^{\frac{3}{2}}v_0\ga^{-1}-v}  g_1\Gamma) \right) dv \frac{1}{t^4}dt \right|\\
    \le &  \int_c^b  \int_{D} 2M_F dv \frac{1}{t^4}dt \\
    = & 2M_F\int_c^b \pi \frac{-G_4^2 H_4^2 t^3+tT^2-\|G_1\ga H_1\|^2}{G_4^2 |\det(H_1)|} \frac{1}{t^4}dt\\
    =& 2M_F\frac{\pi}{G_4^2 |\det(H_1)|} \left[\left(-G_4^2 H_4^2 \log(b)-\frac{T^2}{2b^2}+\frac{\|G_1\ga H_1\|^2}{3b^3}\right)-\left(-G_4^2 H_4^2 \log(c)-\frac{T^2}{2c^2}+\frac{\|G_1\ga H_1\|^2}{3c^3}\right)\right]
\end{align*}

By our choice of $c$ and the bounds for $b$,  after normalization by $V_{\ga,T}[g_1,g_2]=\Theta\left(\frac{T^6}{\|G_1\ga H_1\|^4} \right)$, each one of the six terms in the last line of equation will vanish as $T\to \infty$, provided that $\ga$ is fixed.

Now it remains to estimate the first integral $\int_a^c$.
\begin{align*}
    &\left| \int_a^c  \int_{D} \left(F(\ga^{-1} a_t^{-1} u_{-v} g_1\Gamma) - F(\ga^{-1}a_t^{-1} u_{-t^{\frac{3}{2}}v_0\ga^{-1}-v}  g_1\Gamma) \right) dv \frac{1}{t^4}dt \right|\\
    \le &  C_F \int_a^c  \left( \|-t^{\frac{3}{2}}v_0\ga^{-1}\| \sqrt{\frac{-G_4^2 H_4^2 t^3+tT^2-\|G_1\ga H_1\|^2}{G_4^2 |\det(H_1)|}}\right) \frac{1}{t^4}dt  \tag{by Lemma \ref{lemma on symmetric difference}}\\
    = &  C_F \int_a^c  \left( \|-t^{-\frac{5}{2}}v_0\ga^{-1}\| \sqrt{\frac{-G_4^2 H_4^2 t^3+tT^2-\|G_1\ga H_1\|^2}{G_4^2 |\det(H_1)|}}\right) dt . 
\end{align*}
Observe that for fixed $\ga$, our choice of $c=\frac{\|G_1\ga H_1\|^{1.5}}{T^{1.5}}>>a\asymp \frac{\|G_1\ga H_1\|^{2}}{T^{2}}$ as $T\to \infty$. So the equation above can be bounded from above by (recall that $M_T=\frac{2T^3}{3\sqrt{3}G_4 H_4}$ is the maximum of $-G_4^2 H_4^2 t^3+tT^2$ for $t>0$.)
\begin{align*}
    & (c-a)\cdot C_F\left( \|-a^{-\frac{5}{2}}v_0\ga^{-1}\| \sqrt{\frac{\frac{2T^3}{3\sqrt{3}G_4 H_4}-\|G_1\ga H_1\|^2}{G_4^2 |\det(H_1)|}}\right) \\
    = & O_{F, g_1, g_2}(T^{-1.5})O_{F, g_1, g_2}(T^5\cdot T^{\frac{3}{2}})\\
    = & O_{F, g_1, g_2}(T^5),
\end{align*}
and again, this term goes to zero as $T\to \infty$, after normalization by $V_{\ga,T}[g_1,g_2]=\Theta \left(\frac{T^6}{\|G_1\ga H_1\|^4} \right)$. 
\end{proof}

Then Theorem \ref{unipotent invariance of haar measure} follows from the following elementary lemma from calculus, with $a_{\ga,T}=\phi_{\ga,\mu}(F,T), b_{\ga,T}=\phi_{\ga,\mu}(L_{u_0}F,T)$ and $c_{\ga,T}=\frac{\mu \left(V_{\ga,T}[g_1,g_2] \right)}{\mu \left(H_T[g_1,g_2] \right)}=\Theta_{g_1,g_2}\left(\frac{1}{\|\ga\|^4}\right)$:

\begin{lemma}Let $a_{\ga,T},b_{\ga,T},c_{\ga,T}$ be three bounded (in both $n$ and $T$) real sequences with $\ga \in \SL(2,\Z)$, and $c_{\ga,T}, T\in \R_+$. Assume  
\begin{itemize}
    \item For any $T>0$, we have $\sum_{\ga} c_{\ga,T}=1$;
    \item For any $\e>0$, there exists $N_{\e}$, such that $\sum_{\|\ga\|>N_{\e}} c_{\ga,T}<\e$ for all $T>0$;
    \item For any $\ga$, we have $\lim_{T\to \infty}(a_{\ga,T}-b_{\ga,T})=0$,
\end{itemize}
then $$\lim_{T\to \infty}\sum_{\ga} c_{\ga,T}(a_{\ga,T}-b_{\ga,T})=0.$$
\end{lemma}

\begin{proof}Assume $a_{\ga,T},b_{\ga,T}$ are bounded by $B$. For any $\e>0$, choose $N_{\e}$, such that $\sum_{\|\ga\|>N_{\e}} c_{\ga,T}<\e$ for all $T>0$. Then
\begin{align*}
    \limsup_{T\to \infty}\left| \sum_n c_{\ga,T}(a_{\ga,T}-b_{\ga,T})\right|
    \le&\limsup_{T\to \infty}\left(\left|\sum_{\|\ga\|\le N_{\e}} c_{\ga,T}(a_{\ga,T}-b_{\ga,T})\right|+\left|\sum_{\|\ga\|> N_{\e}} c_{\ga,T}(a_{\ga,T}-b_{\ga,T})\right|\right)\\
    \le & \limsup_{T\to \infty}\left|\sum_{\|\ga\|\le N_{\e}} c_{\ga,T}(a_{\ga,T}-b_{\ga,T})\right|+2B\e\\
    \le& 2B\e
\end{align*}
Letting $\e \to 0$, we are done.
\end{proof}

\section{Proof of Theorem \ref{The $G$-invariance of limiting measure}: The $G$-invariance of limiting measure on $G/\Ga$.} 

Our goal of this section and the next is to generalize Theorem \ref{unipotent invariance of haar measure} from $U$ to $G$. In contrast to the proof of the $U$-invariance in the previous section, our approach here is much more convoluted, which involves the linearization techniques as well as some representation results (cf. \cite{Dani1993LimitDO} and \cite{Shah1996LimitDO}).

To begin with, let $\overline{X} = \overline{G/\Ga}$ denote the one-point compactification of $X=G/\Ga$. Our plan is to show that if $\lim_{T\to \infty} \phi_{\mu}(\cdot,T)=\eta(\cdot)$ (recall this is by applying the Banach-Alaoglu theorem and passing to a subsequence \footnote{recall if any convergent subsequence of a bounded sequence converges to the same limit, then so does the original bounded sequence.}) for some normalized measure $\eta$ on $\overline{X}$, then 
\begin{enumerate}
    \item $\eta(\infty)=0$, and
    \item $\eta$ is $G$-invariant.
\end{enumerate}
As a result we must have $\eta|_X=\mu_X$.

\vspace{3mm}

To prove the first statement, we shall utilize the following theorem due to Shah \cite{Shah1996LimitDO} (also cf. \cite{Dani1993LimitDO}) on the divergence of polynomial trajectories. For a general connected semisimple Lie group $G$ without compact factors and $\Ga$ a lattice in $G$. Let $\frak g$ be the Lie algebra of $G$. For positive integers $d$ and $n$, denote by
$\mathcal {P}_{d,n}(G)$ the set of functions $q: \R^n \to G$ such that for any $a, b \in \R^n$, the map
\begin{equation}
    \tau\in \R \mapsto \Ad(q(\tau a+b))\in \frak{g}
\end{equation}
is a polynomial of degree at most $d$ with respect to some basis of $\frak g$.

Let $V_G = \sum_{i=1}^{\dim \frak g} \bigwedge^i \frak{g}$. There is a natural action of $G$ on $V_G$ induced from the adjoint
representation (in other words, we are considering a representation $\pi:G \to \text{GL}(V_G)$ but sometimes omit the symbol $\pi$). Fix a norm on $V_G$. For a Lie subgroup $H$ of $G$ with Lie algebra $\frak{h}$, take a unit vector $p_H \in \bigwedge^{\dim \frak h} \frak{h}$.

\begin{theorem}[Special case of the theorems 2.1 and 2.2 in \cite{Shah1996LimitDO}, combined]\label{Shah dichotomy theorem}
With notations above, there exist closed subgroups $U_i (i=1,2,...,l)$ such that each $U_i$ is the unipotent radical of a parabolic subgroup, $U_i\Ga$ is compact in $X=G/\Ga$ and for any $d,n\in \N$, $\e,\de>0$, there exists a compact set $C \subset G/\Ga$ such that for any $q \in \mathcal{P}_{d,n}(G)$ and a bounded open convex set $D\subset \R^n$, one of the following holds:
\begin{enumerate}
    \item there exist $\ga'\in \Ga$ and $i=1,...,l$ such that $\sup_{v\in D}\|q(v)\ga'.p_{U_i}\| \le \de$; 
    \item $m(v\in D:q(v)\Ga \notin C)< \e m(D)$, where $m$ is the Lebesgue measure on $\R^n$.
\end{enumerate}
\end{theorem}
\vspace{5mm}


For $g_1\in G=\SL(3,\R)$ and $\ga' \in \Ga$, we shall consider the following family of maps: Let $n=2, t>0$, $v\in \R^2$ and define $q_t(v):=\ga a_t^{-1} u_v^{-1} g_1$ (Here $\ga \in \SL_{2,1}(\Z)$ shouldn't be confused with $\ga'\in \Ga$ above). 

Our strategy is to investigate how $q_t(v)$ fails the condition 1 of the Theorem \ref{Shah dichotomy theorem} by studying the expanding phenomenon of the map $q_t(v)$. As a result, the second statement holds and an approximation argument will give $\eta(\infty)=0$. 

\vspace{5mm}


Denote by $\mathcal{H}_{\Ga}$ the family of all proper closed connected subgroups $H$ of $G$ such that $\Ga \cap H$ is a lattice in H, and $\Ad(H \cap \Ga)$ is Zariski-dense in $\Ad(H)$. We have the following important theorem:

\begin{theorem}[Theorem 1.1 \cite{Ratner91a}, Theorems 2.1 and 3.4 \cite{Dani1993LimitDO}. see also Section 3 and Proposition 4.1 in \cite{Shah1996LimitDO}]\label{ratner's theorem on discreteness}
The set $\mathcal{H}_{\Ga}$ is countable. For any $H\in \mathcal{H}_{\Ga}, \Ga.p_H$ is discrete.    
\end{theorem}

Apply Theorem \ref{Shah dichotomy theorem} with $G=\SL(3,\R)$ to $q_t(v)$. The theorem gives a finite family of unipotent radicals of parabolic subgroups, denoted $U_1,...,U_l$. By Theorem \ref{ratner's theorem on discreteness}, we have that the set $\Ga.p_{U_i}$ is discrete in $V_G$.

Write $V=V_G=V_0 \bigoplus V_1$, where $V_0$ is the space of vectors fixed by $G$ and $V_1$ is its $G$-invariant complement (exists because every finite-dimensional representation of a semisimple Lie group is completely reducible). Denote by $\Pi$ the projection of $V_G$ onto $V_1$ with kernel $V_0$, and note that by Lemma 17 in \cite{Gorodnik2003LatticeAO} that $\Pi(\Ga.p_{U_i})$ is discrete. Since $p_{U_i}$ is not fixed by $G$ (this is because the action is through conjugation and $U_i$'s are not normal subgroups in the simple group $G$), $\Pi(\Ga.p_{U_i})$ does not contain $0$. Furthermore, $\cup_{i=1}^l \Ga.p_{U_i}$ is discrete. So it follows that 
\begin{equation}\label{inf of the projection is at least r}
    \inf_{x\in \cup_{i=1}^l \Ga.p_{U_i}} \|\Pi(x)\|:=r>0.
\end{equation}
The union $\cup_{i=1}^l
\Ga.p_{U_i}$ is what we are looking for in the first outcome of Theorem \ref{Shah dichotomy theorem}. The following lemma will play a crucial role in proving $\eta(\infty)=0$:

\begin{lemma}\label{our lemma expansion inequality for applying shah dichotomy} Let $U_1,...,U_l$ are as above. Fix $(a_1,a_2)\in \R^2$ and $\beta=\frac{1}{2}t^{0.4}$. Then for any $\de>0$, there exist $t_0>0$ such that for any $0<t<t_0$, 
\begin{equation}\label{expansion inequality for applying shah dichotomy}
    \sup_{v\in D(\beta)+(a_1,a_2)} \|a_t^{-1}u_v^{-1}g_1.x\| >\de,
\end{equation}
for any $x\in \cup_{i=1}^r \Ga.p_{U_i}$,
where $D(\beta)$ is the disk of radius $\beta$ in $\R^2$ centered at the origin.
\end{lemma}

\subsection{Proof of Lemma \ref{our lemma expansion inequality for applying shah dichotomy}}

The idea of the proof is to decompose the elements in $AU$ into elements in two subgroups of $\SL(3,\R)$ isomorphic to $\SL(2,\R)$ and then use the representation theory of $\SL(2,\R)$ and Lemma \ref{very important inequaity lemma of nimish-gorodnik} below to show find the expansion property for each. Note that a similar approach was used in \cite{Kleinbock2006DIRICHLETSTO}.

Observe that
\begin{align}
    a_t^{-1}u_v^{-1}
=&
\begin{bmatrix}
\sqrt t & 0 & 0 \\
0 & \sqrt t & 0 \\
0 & 0 & \frac{1}{t} 
\end{bmatrix}
\begin{bmatrix}
1 & 0 & 0 \\
0 & 1 & 0 \\
-z & -y & 1 
\end{bmatrix} \nonumber\\ 
=&
\begin{bmatrix}
\sqrt t & 0 & 0 \\
0 & 1 & 0 \\
0 & 0 & \frac{1}{\sqrt t} 
\end{bmatrix}
\begin{bmatrix}
1 & 0 & 0 \\
0 & 1 & 0 \\
-\frac{y}{\sqrt{t}} & 0 & 1 
\end{bmatrix}
\begin{bmatrix}
1 & 0 & 0 \\
0 & \sqrt t & 0 \\
0 & 0 & \frac{1}{\sqrt{t}} 
\end{bmatrix}
\begin{bmatrix}
1 & 0 & 0 \\
0 & 1 & 0 \\
0 & -y & 1 
\end{bmatrix} \label{one order}\\
(alternatively)=&
\begin{bmatrix}
1 & 0 & 0 \\
0 & \sqrt t & 0 \\
0 & 0 & \frac{1}{\sqrt{t}} 
\end{bmatrix}
\begin{bmatrix}
1 & 0 & 0 \\
0 & 1 & 0 \\
0 & -\frac{y}{\sqrt{t}} & 1 
\end{bmatrix}
\begin{bmatrix}
\sqrt t & 0 & 0 \\
0 & 1 & 0 \\
0 & 0 & \frac{1}{\sqrt t} 
\end{bmatrix}
\begin{bmatrix}
1 & 0 & 0 \\
0 & 1 & 0 \\
-z & 0 & 1 
\end{bmatrix}\label{the other order}
\end{align}

It follows that the action of $a_t^{-1}u_v^{-1}$ on a vector through can be identified with two consecutive actions of elements in 

\begin{equation}
\SL^{(1)}(2,\R):=
\left \{\begin{bmatrix}
1 & 0 & 0 \\
0 & a & b \\
0 & c & d 
\end{bmatrix}:ad-bc=1 \right \}
\end{equation}
and
\begin{equation}
    \SL^{(2)}(2,\R):=
\left \{\begin{bmatrix}
a & 0 & b \\
0 & 1 & 0 \\
b & 0 & d 
\end{bmatrix}:ad-bc=1 \right \},
\end{equation}
each isomorphic to $\SL(2,\R)$. For example $\begin{bmatrix}
1 & 0 & 0 \\
0 & \sqrt t & 0 \\
0 & 0 & \frac{1}{\sqrt{t}} 
\end{bmatrix}$ can be identified with $\begin{bmatrix}
\sqrt t & 0 \\
0 & \frac{1}{\sqrt{t}} 
\end{bmatrix}$ and $\begin{bmatrix}
1 & 0 & 0 \\
0 & 1 & 0 \\
-z & 0 & 1 
\end{bmatrix}$
can be identified with
$\begin{bmatrix}
1  & 0 \\
-z & 1 
\end{bmatrix}$.

\vspace{5mm}

Now we return to the study of the study of $a_t^{-1}u_v^{-1}w$, where $w=g_1.x$ with $x\in \cup_{i=1}^l \Ga.p_{U_i}\subset{V_G}$ as above. We consider the two following decompositions of $V_1$ (recall that this is the invariant completement of $V_0$ of fixed vectors in $V$ by $G$),
\begin{align}
    V_1 &= V_{0}^{(1)} \oplus V_{1}^{(1)} \\
      &= V_{0}^{(2)} \oplus V_{1}^{(2)} \label{decompositive of vs w.r.t. SL2}
\end{align}
where $V_0^{(i)}$ is the subspace of fixed vectors of the action of $\SL^{(i)}(2,\R)$ for $i=1,2$ and $V_1^{(i)}$ is its invariant complement. Note that since $G$ is generated by $\SL^{(1)}(2,\R)$ and $\SL^{(2)}(2,\R)$, see Lemma \ref{sl2 generating sl3}, we must have\footnote{Alternatively, note that the group $UA$ is epimorphic, see e.g. \cite{Shah2000OnAO}.} $V_{0}^{(1)} \cap V_{0}^{(2)}=\{0\}$. Later on, we will also decompose $V_{1}^{(i)}$ into irreducible components with respect to the restriction of the representation $\pi$ to $\SL^{(i)}(2,\R)$, respectively.

\vspace{5mm}
For any $w_1\in V_1, \|w_1\|\ge r$ for some $r>0$, we write 
\begin{equation}    w_1=w_{0}^{(1)}+w_{1}^{(1)}=w_{0}^{(2)}+w_{1}^{(2)}
\end{equation}
where $w_{i}^{(j)} \in V_{i}^{(j)}$ from above. We expect the non-fixed components, $w_{1}^{(1)}$ and $w_{1}^{(2)}$ to be not simultaneously small. In fact we have:

\begin{claim}
$\inf_{w_1\in V_1, \|w_1\|\ge r} \left( \|w_{1}^{(1)}\|+\|w_{1}^{(2)}\| \right) = Lr$ for some $L>0$
\end{claim}
\begin{proof}[Proof of the claim]\renewcommand{\qedsymbol}{\ensuremath{\#}}
Indeed, using the fact that the sphere of radius $r$ is compact and that $V_{0}^{(1)} \cap V_{0}^{(2)}=\{0\}$, we have $$\inf_{w_1\in V_1, \|w_1\|= 1} \left( \|w_{1}^{(1)}\|+\|w_{1}^{(2)}\| \right) = L > 0.$$ 
By scaling, it is easy to see for $r \ge 1$, 
\begin{align*}
    \inf_{w_1\in V_1, \|w_1\|= r} \left( \|w_{1}^{(1)}\|+\|w_{1}^{(2)}\| \right) 
    = & r \inf_{w_1\in V_1, \|w_1\|= 1} \left( \|w_{1}^{(1)}\|+\|w_{1}^{(2)}\| \right) \tag{by linearity and scaling} \\
    =& Lr>0.
\end{align*}
\end{proof}

It follows from this claim that for any vector that cannot be fixed by $G$, then its non-$\SL^{(i)}(2,\R)$-fixed component must be bounded away from zero by at least $\frac{Lr}{2}>0$ for either $i=1$ or $i=2$. 

\vspace{5mm}
Therefore, the study of the expanding phenomenon on $a_t^{-1}u_v^{-1}w$ can be transformed into the study of representation theory of $\SL(2,\R)$. Recall that if $\pi$ is an $(n+1)$-dimensional irreducible representation of $\SL(2,\R)$ and $\pi'$ is the induced Lie algebra representation, then there exist a basis $v_0,v_1,...,v_n$ such that
\begin{align}
    &\pi'(H)(v_i)=(n-2i)v_i, i=0,1,...,n;\\
    &\pi'(X)(v_i)=i(n-i+1)v_{i-1}, i=0,1,...,n;\\
    &\pi'(Y)(v_i)=v_{i+1}, i=0,1,...,n~(v_{n+1}=0).
\end{align}
where
$H = \begin{bmatrix}
    1 & ~~0\\
    0 & -1
  \end{bmatrix},
  X = \begin{bmatrix}
    0 & 1\\
    0 & 0
  \end{bmatrix}$, and
  $Y = \begin{bmatrix}
    0 & 0\\
    1 & 0
  \end{bmatrix}$ form a generating set of $\frak{sl}(2,\R)$.

\vspace{5mm}
 Under the matrix Lie group-Lie algebra correspondence, 
\[\pi \left(\begin{bmatrix}
    \sqrt{t} & 0\\
    0 & \frac{1}{\sqrt{t}}
\end{bmatrix} \right)=\pi(\exp(\log(\sqrt{t})H))\] has the matrix representation 
\begin{equation}\label{matrix rep of a_t under the basis}
    \begin{bmatrix}
    \sqrt{t}^n & & &\\
    & \sqrt{t}^{n-2} & & \\
    & & \ddots &\\
    & & & \sqrt{t}^{-n}
  \end{bmatrix}
\end{equation}
and 
\[\pi \left(\begin{bmatrix}
    1 & 0\\
    -y & 1
\end{bmatrix} \right)=\pi(\exp(-yY))\]
has the matrix representation
\begin{equation}\label{matrix rep of u_y under the basis}
    \begin{bmatrix}
    1 & & &\\
    p_{21}(y) & 1 & & \\
    \vdots   & \ddots & \ddots &\\
    p_{nn}(y) & \cdots &  p_{n1}(y)  &  1
  \end{bmatrix}
\end{equation}
where $p_{kl}(y)$ is a monomial in $y$ of degree $l$.
Both matrices are under the basis $v_0,v_1,...,v_n$. Observe that $\R v_n$ is the fixed subspace of $\pi \left(\begin{bmatrix}
    1 & 0\\
    -y & 1
\end{bmatrix}\right)$.

\vspace{5mm}
The following important lemma allows to prove the expansion of vectors in the representation space by the group $UA$, see \eqref{expansion inequality}.

\begin{lemma}[\cite{Gorodnik2003LatticeAO} Lemma 13, \cite{Shah1996LimitDO} Lemma 5.1 ]\label{very important inequaity lemma of nimish-gorodnik}
Let $V$ be a finite dimensional vector space with a norm $\|\cdot\|$, $\frak{n}$ be a nilpotent Lie algebra of $\End(V)$ with a basis $\{b_i:i=1,...,m\}$, and $N=\exp(\frak{n})\subset \GL(V)$ be the Lie group of $\frak{n}$. Define a map $\Theta:\R^m \to N$:
\begin{equation}
    \Theta(t_1,...,t_m)=\exp \left( \sum_{i=1}^m t_ib_i \right), (t_1,....,t_m)\in \R^m.
\end{equation}
This is a polynomial map since $N$ is nilpotent.


For $\beta>0$, put 
$D_{\exp}(\beta)=\Theta([0,\beta]\times \cdots \times [0,\beta])$. 
Let
\begin{equation}
    W=\{v\in V:\frak{n}.v=0\}=\{v\in V:N.v=v\}.
\end{equation}
Denote by $\pr_W$ the orthogonal projection on $W$ with respect to some scalar product on $V$. Then there exists a constant $c_0>0$ such that for any $\beta \in (0,1)$ and $v\in V$,
\begin{equation}
    \sup_{n\in D_{\exp}(\beta)}\|\pr_W(n.v)\| \ge c_0 \beta^d\|v\|,
\end{equation}
where $d$ is the degree of the polynomial map $\Theta$.
\end{lemma}

In our case, $V=V_G$, $N=
\pi \left(\begin{bmatrix}
    1 & 0\\
    \R & 1
\end{bmatrix}\right)$ (identified with $\pi \left(\begin{bmatrix}
    1 & 0 & 0\\
    0 & 1 & 0\\
    0 &\R & 1
\end{bmatrix}\right))$ or $\pi \left(\begin{bmatrix}
    1 & 0 & 0\\
    0 & 1 & 0\\
    \R & 0 & 1
\end{bmatrix}\right))$, $m=1$, $d=n=\dim V$ and $W=\R v_n$. Note that (keep in mind that for convenience we omit the representation symbol $\pi$):

\begin{align}
& \left \|\begin{bmatrix}
\st & 0 \\
0 & \frac{1}{\st} 
\end{bmatrix}
\begin{bmatrix}
1 & 0 \\
-y & 1 
\end{bmatrix}.w_1 \right \| \nonumber \\
\ge & C \left \| \pr_W\left(\begin{bmatrix}
\st & 0 \\
0 & \frac{1}{\st} 
\end{bmatrix}
\begin{bmatrix}
1 & 0 \\
-y & 1 
\end{bmatrix}.w_1 \right) \right \| \tag{by the boundedness of the linear operator $\pr_W$}\\
= & C \left \| \begin{bmatrix}
\st & 0 \\
0 & \frac{1}{\st} 
\end{bmatrix}.
 \pr_W\left(
\begin{bmatrix}
1 & 0 \\
-y & 1 
\end{bmatrix}.w_1 \right) \right \| \tag{since the matrix \eqref{matrix rep of a_t under the basis} commutes with $\pr_W=\pr_{\R e_n}$}\\
=& C \frac{1}{\st^n} \left \|
 \pr_W\left(
\begin{bmatrix}
1 & 0 \\
-y & 1 
\end{bmatrix}.w_1 \right) \right \| \tag{see the last component of \eqref{matrix rep of a_t under the basis}}\\
=& C \frac{1}{\st^n} \left \|
 \pr_W\left(
\begin{bmatrix}
1 & 0 \\
-(y-a_1) & 1 
\end{bmatrix}\begin{bmatrix}
1 & 0 \\
-a_1 & 1 
\end{bmatrix}.w_1 \right) \right \|
\end{align}
By Lemma \ref{very important inequaity lemma of nimish-gorodnik}, for any $a_1\in \R$, we have that with $\beta=\frac{1}{2}t^{0.4}$, there exists $C'>0$ and $-y\in [a_1,a_1+\beta]$
\begin{align}
    \left \|
 \pr_W\left(
\begin{bmatrix}
1 & 0 \\
-y & 1 
\end{bmatrix}.w_1 \right) \right \| \ge C'\beta^n\|w_1\|.
\end{align}
Namely there exist $C'>0$ such that for some $-y\in [a_1,a_1+\beta]$
\begin{align}\label{expansion inequality}
    \left \|\begin{bmatrix}
\st & 0 \\
0 & \frac{1}{\st} 
\end{bmatrix}
\begin{bmatrix}
1 & 0 \\
-y & 1 
\end{bmatrix}.w_1 \right \| \ge C' \frac{1}{t^{0.1n}}\|w_1\|.
\end{align}

In general, if the $\SL^{(i)}(2,\R)$-representation is not irreducible, we may choose a basis for the non-fixed component $V_{1}^{(i)}$ and decompose it into irreducible representations. By taking the sup-norm with respect to the coefficients of a vector under this basis, we can again recover the inequality \eqref{expansion inequality}.

Now for $w=g_1.x$, under \eqref{inf of the projection is at least r} we have $\|w\|\ge \|g_1^{-1}\|^{-1}_{op}\|x\|\ge C''\|g_1^{-1}\|^{-1}_{op} \Pi(x) \ge C''\|g_1^{-1}\|^{-1}_{op} r,$ and the claim above guarantees that we have such an expansion for $a_t^{-1}u_v^{-1}.w$ by choosing either the order \eqref{one order} or the order \eqref{the other order} for the decomposition of $a_t^{-1}u_v^{-1}$. 

\subsection{Proof of $\eta{(\infty)}=0$}\label{section 4.2}

\vspace{5mm}
Recall
\begin{align}
    \phi_{\ga,\mu}(F,T):=& \frac{1}{\mu \left(V_{\ga,T}[g_1,g_2] \right)}\int_{V_{\ga,T}[g_1,g_2]} F(h^{-1}g_1\Gamma) d\mu(h)\\
    =& \frac{1}{\mu \left(V_{\ga,T}[g_1,g_2] \right)} \int_{a}^b \int_{D} F \left(\ga^{-1} a_t^{-1} u_v^{-1} g_1\Gamma \right) dv \frac{1}{t^4}dt
\end{align}
where 
\begin{equation}\label{definition of D}
    D=D_{\ga, t, T}:=\{v\in \R^2: \|G_3\ga H_1+G_4 v\ga H_1+G_4 H_3 t^{\frac{3}{2}}\|^2 \le -G_4^2 H_4^2 t^3+tT^2-\|G_1\ga H_1\|^2\},
\end{equation}
and $0<a<b$ are two positive roots of $-G_4^2 H_4^2 t^3+tT^2-\|G_1\ga H_1\|^2:=-Bt^3+tT^2-f_1(\ga)$.


\vspace{5mm}
 $D$ is the elliptic disk centered at $-G_4^{-1}G_3 -H_3H_1^{-1}\ga^{-1}t^{\frac{3}{2}} \in \R^2$ and it contains the disk with the same center whose radius is equal to the \textit{shorter radius} of the ellipse, which is $$C_{g_1,g_2,\ga}\sqrt{ -G_4^2 H_4^2 t^3+tT^2-\|G_1\ga H_1\|^2},$$
 for some constant $C_{g_1,g_2,\ga}>0$.
 
 The  displacement of the center from the origin has two parts: one part ($-G_4^{-1}G_3$) is constant and the other part ($-H_3H_1^{-1}\ga^{-1}t^{\frac{3}{2}}$) involves $t$ for which we have to handle separately. 

Our next step is to find a sub-interval of $ (a,b)$, in which we have $D \supset D(2\beta)-G_4^{-1}G_3=D(t^{0.4})-G_4^{-1}G_3$ for sufficiently large $T$ and in the complement of that sub-interval, the integral above will vanish as $T\to \infty$ (Here $D(2\beta)\subset \R^2$ is the disk centered at O with radius $2\beta$ which shouldn't be confused with $D_{\exp}$ from Lemma \ref{very important inequaity lemma of nimish-gorodnik}). Using the elementary geometry fact that  if there are two disks, then one disk contains the other if the difference of their radii is greater than the distance between their centers, and simplifying coefficients, we are actually interested in finding $t$ satisfying
\begin{equation}
    \sqrt{-Bt^3+tT^2-f_1(\ga)} \gtrsim
 c_1t^{\frac{3}{2}}+2\beta= c_1t^{\frac{3}{2}}+t^{0.4},
\end{equation}
for sufficiently large $T$.

Take $t=a+\de$ with $\de>0$. Since $a\asymp \frac{1}{T^2}$ is a root of $-Bt^3+tT^2-f_1(\ga)$, the equation above is equivalent to
\begin{equation}\label{important inequality determininig the range for t}
    \de T^2 \gtrsim
 B(3a^2\de+3a\de^2+\de^3)+(c_1(a+\de)^{\frac{3}{2}}+(a+\de)^{0.4})^2.
\end{equation}

Recall from \eqref{the truncation c} we used the truncation 
\begin{equation}
    c=c_{\ga,T}=\frac{\|G_1\ga H_1\|^{1.5}}{T^{1.5}}
\end{equation}
for the integral (this will serve as the upper bound below). Take
\begin{equation}\label{definition of delta T}
    \de_T=\frac{\|G_1\ga H_1\|^{2.5}}{T^{2.5}}.
\end{equation}
It is easy to see that for $t\in (a+\de_T, c)$, \eqref{important inequality determininig the range for t} holds as $T\to \infty$, since the left hand side is asymptotically $\frac{1}{T^{0.5}}$. Moreover,
\begin{align}
  & \lim_{T\to \infty}\left|\frac{1}{\mu \left(V_{\ga,T}[g_1,g_2] \right)} \int_{a}^{a+\de_T} \int_{D} F \left(\ga^{-1} a_t^{-1} u_v^{-1} g_1\Gamma \right) dv \frac{1}{t^4}dt \right| \nonumber \\
  \le & \lim_{T\to \infty} \frac{1}{\mu \left(V_{\ga,T}[g_1,g_2] \right)} \int_{a}^{a+\de_T} \pi \frac{-G_4^2 H_4^2 t^3+tT^2-\|G_1\ga H_1\|^2}{G_4^2 |\det(H_1)|}  \frac{1}{t^4}dt \nonumber \\
  =& 0 \tag{integrate three parts separately and recall $\mu \left(V_{\ga,T}[g_1,g_2] \right) \asymp_{\ga} T^6$}.
\end{align}

and by our computation in Section 3, we also have
\begin{align}
  \lim_{T\to \infty} \frac{1}{\mu \left(V_{\ga,T}[g_1,g_2] \right)} \int_{c}^{b} \int_{D} F \left(\ga^{-1} a_t^{-1} u_v^{-1} g_1\Gamma \right) dv \frac{1}{t^4}dt=0.
\end{align}

\vspace{5mm}
For the integration in the range between $a+\de_T$ and $c$, noticing that by Lemma \ref{our lemma expansion inequality for applying shah dichotomy} with $(a_1,a_2)=-G_4^{-1}G_3$, 
the first outcome in Theorem \ref{Shah dichotomy theorem} fails as $T\to \infty$ and therefore the second outcome holds. Here $D \supset D(\beta)-G_4^{-1}G_3$ is the open convex set involved. Let $\epsilon>0$, and let $C\subset G/\Ga$ be a compact set such that
\begin{equation*}
    m(v\in D:q_t(v)\Ga \notin C)< \e m(D),
\end{equation*}
whenever $t\in (a+\de_T,c)$.
Taking the complement,
\begin{equation}\label{the second outcome of Shah theorem}
    m(v\in D:q_t(v)\Ga \in C) > (1-\e) m(D).
\end{equation}
Now take $F\in C_c(G/\Ga)$ with $\mathbf{1}_C \le F \le 1$. As $T\to \infty$,
\begin{align*}
&\frac{1}{\mu \left(V_{\ga,T}[g_1,g_2] \right)}\int_{V_{\ga,T}[g_1,g_2]} F( h^{-1}g_1\Gamma) d\mu(h)\\
\ge& \frac{1}{\mu \left(V_{\ga,T}[g_1,g_2] \right)} \int_{a+\de_T}^{c} \int_{D} F \left(q_t(v)\Ga \right) dv \frac{1}{t^4}dt\\
\ge& \frac{1}{\mu \left(V_{\ga,T}[g_1,g_2] \right)} \int_{a+\de_T}^{c} \int_{D} \mathbf{1}_C \left(q_t(v)\Ga \right) dv \frac{1}{t^4}dt \\
\ge& \frac{1}{\mu \left(V_{\ga,T}[g_1,g_2] \right)} \int_{a+\de_T}^{c} \int_{D} (1-\e)m(D) \frac{1}{t^4}dt \tag{by \eqref{the second outcome of Shah theorem}}
\end{align*}

By the computation above we know
\begin{align*}
    &\lim_{T\to \infty}\frac{1}{\mu \left(V_{\ga,T}[g_1,g_2] \right)} \int_{a+\de_T}^{c} \int_{D} (1-\e)m(D) \frac{1}{t^4}dt \\
    =&  \lim_{T\to \infty}\frac{1}{\mu \left(V_{\ga,T}[g_1,g_2] \right)} \int_{a}^{b} \int_{D} (1-\e)m(D) \frac{1}{t^4}dt \\
    =& 1-\e
\end{align*}
Therefore,
\begin{equation}
    \liminf_{T\to \infty}\frac{1}{\mu \left(V_{\ga,T}[g_1,g_2] \right)}\int_{V_{\ga,T}[g_1,g_2]} F( h^{-1}g_1\Gamma) d\mu(h)\ge 1-\e
\end{equation}
and thus
\begin{equation}
    \limsup_{T\to \infty}\frac{1}{\mu \left(V_{\ga,T}[g_1,g_2] \right)}\int_{V_{\ga,T}[g_1,g_2]} \mathbf{1}_{\Supp (F)^c}( h^{-1}g_1\Gamma) d\mu(h)\le \e
\end{equation}

Recall that each coefficient $\frac{\mu \left(V_{\ga,T}[g_1,g_2] \right)}{\mu \left(H_T[g_1,g_2] \right)} \asymp \frac{1}{\|\ga\|^4}$ in the convex linear combination, and therefore by \eqref{head estimate for the sum over SL2} the sum is absolutely convergent. Choose $A_{\e}>0$ so that 
\begin{equation*}
    \sum_{\substack{\ga\in \SL(2,\Z),\\ \|G_1\ga H_1\|^2<M_T},\\
    \|\ga\|>A_{\e}}\frac{\mu \left(V_{\ga,T}[g_1,g_2] \right)}{\mu \left(H_T[g_1,g_2] \right)}\cdot \frac{1}{\mu \left(V_{\ga,T}[g_1,g_2] \right)}\int_{V_{\ga,T}[g_1,g_2]} \mathbf{1}_{\Supp (F)^c}(h^{-1}g_1\Gamma) d\mu(h)<\e
\end{equation*}
Hence,
\begin{align*}
     \eta(\infty) \le & \limsup_{T\to \infty}\phi_{\mu}(\mathbf{1}_{\Supp (F)^c},T) \\
\le & \limsup_{T\to \infty} \sum_{\substack{\ga\in \SL(2,\Z),\\ \|G_1\ga H_1\|^2<M_T\\ \|\ga\|\le A_{\e}}}\frac{\mu \left(V_{\ga,T}[g_1,g_2] \right)}{\mu \left(H_T[g_1,g_2] \right)}\cdot \frac{1}{\mu \left(V_{\ga,T}[g_1,g_2] \right)}\int_{V_{\ga,T}[g_1,g_2]} \mathbf{1}_{\Supp (F)^c}(h^{-1}g_1\Gamma) d\mu(h)+\e\\
\le & 2\e.
\end{align*}
Therefore $\eta(\infty)=0$.

\subsection{Proof of $G$-invariance}

Recall $U=\begin{bmatrix}
    I_2 & 0\\
    \R^2 & 1
\end{bmatrix}$. For a closed subgroup
$H$ of $G$, denote
\begin{align}
    &N(H,U):=\{g\in G:Ug\subset gH\},\\
    &S(H,U):=\cup_{H'\subsetneq H, H'\in \mathcal{H}_{\Ga}}N(H',U).
\end{align}
Consider,
\begin{equation}
    Y:=\bigcup_{H\in \mathcal{H}_{\Ga}}N(H,U)\Ga=\bigcup_{H\in \mathcal{H}_{\Ga}}[N(H,U)-S(H,U)]\Ga \subset G/\Ga,
\end{equation}
where $\mathcal{H}_
{\Ga}$ was defined above Theorem \ref{ratner's theorem on discreteness}.
The equality holds since for any $g\in N(H,U)$, if $g$ is also in $S(H,U)$, then $g$ must belong to $N(H',U)$ for some $H'\subsetneq H$ (note for Lie subgroups, this condition means $\dim H' < \dim H$) and $H'\in \mathcal{H}_{\Ga}$. Since $H'$ has strictly lower dimension, by repeating this argument we see eventually, $g$ will fall into some $N(\tilde{H}, U)-S(\tilde{H}, U)$ (with $S(\tilde{H}, U)$ possibly empty when $\tilde{H}$ has minimal dimension).

\begin{remark}
It follows from Dani’s generalization of Borel density theorem and Ratner’s topological
rigidity that $Y$ precisely form the set of points in $G/\Ga$ with non-dense orbit (Lemma 19.4 \cite{Starkov1996}).
    Specifically, $y \in Y$ if and only if the orbit $Uy$ is not dense in $G/\Ga$.
\end{remark}

A proof similar to that in the previous section shows that $Y$ is also a $\eta$-null set:
\begin{lemma}
    $\eta(Y)=0$.
\end{lemma}

\begin{proof}
    By the discreteness of $\mathcal{H}_{\Ga}$ (Theorem \ref{ratner's theorem on discreteness}), it suffices to show
    \begin{equation}
        \eta([N(H,U)-S(H,U)]\Ga)=0,
    \end{equation}
    for all $H\in \mathcal{H}_{\Ga}$. Since $[N(H,U)-S(H,U)]\Ga$ is a countable union of compact subsets in $G/\Ga$ (See \cite{Mozes1995OnTS} Proposition 3.1), it suffices to show $\eta(C)=0$ for any compact subset $C$ of $[N(H,U)-S(H,U)]\Ga$.

The main tool in the proof is the following consequence of Prososition 5.4 in \cite{Shah1994LimitDO}:

\begin{theorem}\label{second dichotomy theorem of Shah}
    Let $d,n \in \N, \e>0, H\in \mathcal{H}_{\Ga}$. For any compact set $C \subset [N(H, U)-S(H,U)]\Ga$, there exists a compact set $F \subset V_G$ such that for any neighborhood $\Phi$ of $F$ in $V_G$, there exists a
neighborhood $\Psi$ of $C$ in $G/\Ga$ such that for any $q \in \mathcal{P}_{d,n}(G)$ and a bounded open convex set
$D \subset \R^n$, one of the following holds:
\begin{enumerate}
    \item There exist $\ga\in \Ga$ such that $q(D)\ga.p_H \subset \Phi$.
    \item $m(t\in D: q(t)\Ga \in \Psi)<\e m(D)$, where $m$ is the Lebesgue measure on $\R^n$.
\end{enumerate}
\end{theorem}
Fix $\e>0$. Now apply Theorem \ref{second dichotomy theorem of Shah} to $q_t(v)$, which gives a compact set $F\subset V_G$. Let $\Phi$ be a compact neighborhood of $F$ in $V_G$ and $\Psi \supset C$ be the neighborhood of $C$ given by the theorem. By the same argument as in the Section 4.3 (with Lemma \ref{our lemma expansion inequality for applying shah dichotomy} applied), we have that for $t\in (a+\de_T, c)$,
\begin{equation*}
    m(v\in D:q_t(v)\Ga \in \Psi)<\e m(D),
\end{equation*}
whenever $t\in (a+\de_T,c)$. Here $D$ is the elliptic disc defined in \eqref{definition of D}. $c$ and $\delta_T$ are defined in \eqref{the truncation c} and \eqref{definition of delta T}, respectively.

Now let $f\in C_c(G/\Ga)$ be such that $\mathbf{1}_C \le f \le 1$ and $\Supp(f)\subset \Psi$, it follows that
\begin{align*}
& \limsup_{T\to \infty}\frac{1}{\mu \left(V_{\ga,T}[g_1,g_2] \right)}\int_{V_{\ga,T}[g_1,g_2]} f( h^{-1}g_1\Gamma) d\mu(h)\\
=& \limsup_{T\to \infty} \frac{1}{\mu \left(V_{\ga,T}[g_1,g_2] \right)} \int_{a}^{b} \int_{D} f \left(q_t(v)\Ga \right) dv \\
=& \limsup_{T\to \infty} \frac{1}{\mu \left(V_{\ga,T}[g_1,g_2] \right)} \int_{a+\de_T}^{c} \int_{D} f \left(q_t(v)\Ga \right) dv \frac{1}{t^4}dt\\
\le& \limsup_{T\to \infty} \frac{1}{\mu \left(V_{\ga,T}[g_1,g_2] \right)} \int_{a+\de_T}^{c} \int_{D} \mathbf{1}_{\Psi} \left(q_t(v)\Ga \right) dv \frac{1}{t^4}dt \\
\le& \limsup_{T\to \infty} \frac{1}{\mu \left(V_{\ga,T}[g_1,g_2] \right)} \int_{a+\de_T}^{c} \int_{D} \e m(D) \frac{1}{t^4}dt \\
\le & \e,
\end{align*}
and by the argument of truncating the convex linear combination as we did in the end of subsection \ref{section 4.2}.


\begin{align*}
\eta(C)
\le & \limsup_{T\to \infty} \phi_{\mu}(\mathbf{1}_{C},T) \\
\le & \limsup_{T\to \infty} \sum_{\substack{\ga\in \SL(2,\Z),\\ \|G_1\ga H_1\|^2<M_T\\ \|\ga\|\le A_{\e}}}\frac{\mu \left(V_{\ga,T}[g_1,g_2] \right)}{\mu \left(H_T[g_1,g_2] \right)}\cdot \frac{1}{\mu \left(V_{\ga,T}[g_1,g_2] \right)}\int_{V_{\ga,T}[g_1,g_2]} \mathbf{1}_{C}(h^{-1}g_1\Gamma) d\mu(h)+\e\\
\le & 2\e.
\end{align*}
Hence $\eta(C)=0$ and thus $\eta(Y)=0$.
\end{proof}

Now we perform the ergodic decomposition of the $U$-invariant measure $\eta$. By Theorem 2.2 of \cite{Mozes1995OnTS}, each ergodic component of $\eta$ is either $G$-invariant or supported on $Y\cup \{\infty\}$. But we have just prove that $\eta(Y\cup \{\infty\})=0$. Therefore $\eta$ is $G$-invariant.

\section{Proof of Theorem \ref{equidistribution result on  G mod H}: The limiting measure on $H\backslash G$.}

We follow Section 2.5 of \cite{Gorodnik2004DistributionOL}.
First, we provide an explicit measurable section $$\sigma:H\backslash G \to Y\subset G,$$and then we define a measure $\nu_Y$ on $Y$ such that \begin{equation} \label{unfolding haar measure on G using section}
    dm=d\mu d\nu_Y,
\end{equation} where $\mu$ is the left Haar measure on $H$ and $m$ is a Haar measure on $G$.

To define the section,  let $\mathcal{F}_2 \subset \begin{bmatrix}
\SL(2,\R) & 0\\
0 & 1
\end{bmatrix}$ denote a measurable fundamental domain of 
\begin{equation}
   \begin{bmatrix}
    \SL(2,\Z) & 0\\
    0 & 1
    \end{bmatrix} \bigg\backslash 
   \begin{bmatrix}
    \SL(2,\R) & 0\\
    0 & 1
    \end{bmatrix} \bigg\slash  \begin{bmatrix}
    \SO(2,\R) & 0\\
    0 & 1
    \end{bmatrix} 
    \cong
    \SL(2,\Z)\backslash \SL(2,\R) /\SO(2,\R),
\end{equation}
and consider $$Y:=\mathcal{F}_2\cdot \SO(3,\R).$$

Then, we claim that the product map $H \times Y \to G$ is a Borel isomorphism, which defines a section $\sigma$ identifying $H\backslash G$ with $Y$. The surjectivity is clear from the block-wise Iwasawa decomposition. We only verify the injectivity here:

Suppose $h_1g_1s_1=h_2g_2s_2$ where $h_i\in H, g_i\in \mathcal{F}_2$ and $s_i\in \SO(3,\R)$. Then 
$$\begin{bmatrix}
* & 0 \\
* & *
\end{bmatrix}\ni g_2^{-1} h_2^{-1}h_1 g_1=s_2 s_1^{-1} \in \SO(3,\R),$$
and it follows from the definition of $\SO(3,\R)$ that both side must lie in $\begin{bmatrix}\SO(2,\R) & 0\\0 & 1 \end{bmatrix}$. So $h_2^{-1}h_1=g_2s_2s_1^{-1}g_1^{-1}\in \begin{bmatrix}\SL(2,\R) & 0\\0 & 1 \end{bmatrix}$. But $\begin{bmatrix}\SL(2,\R) & 0\\0 & 1 \end{bmatrix} \cap H= \begin{bmatrix}\SL(2,\Z) & 0\\0 & 1 \end{bmatrix}$. Hence $\ga g_1s_1=g_2s_2$ for some $\ga \in \begin{bmatrix}
\SL(2,\Z) & 0 \\0 & 1
\end{bmatrix}$. It follows that $g_2^{-1}\ga g_1=s_2 s_1^{-1}\in \begin{bmatrix}\SL(2,\R) & 0 \\0 & 1 \end{bmatrix} \cap \SO(3,\R)=\begin{bmatrix}\SO(2,\R) & 0 \\0 & 1 \end{bmatrix}$. Hence $g_1=g_2, s_1=s_2$ by the definition of fundamental domain. and $h_1=h_2$.


Next, let $T:=
\left\{ \begin{bmatrix}
t^{-\frac{1}{2}}\SL(2,\R) & 0\\
\R^2 & t
\end{bmatrix}:t>0 \right\}$ and $S:=\SO(3,\R)$. On $T$ we define the Haar measure in a similar way as the left Haar measure on $H$ (cf. Proposition \ref{Haar measure on $H$}). Below $d\tau=dv \frac{dt}{t^4} dg$, where $dg$ is the Haar measure on $\SL(2,\R)$ defined through standard Iwasawa decomposition, under which $\SO(2,\R)$ has volume $2\pi$. $d\rho$ is the standard Haar measure on $\SO(3,\R)$ that gives the total volume $8\pi^2$. Then, by using Theorem 8.32 of \cite{KN02}, we unfold a Haar measure on $G$ as follows
\begin{align}
    &\int_G f(g)dm(g)\\
   =&\int_S \int_T f(\tau \rho) d\tau d\rho \\
   =&\int_S\int_{\SL(2,\R)} \int_0^{\infty}\int_{\R^2} f(u_v a_t g \rho)    dv \frac{dt}{t^4}dg d\rho\\
   =&\int_S\sum_{\ga \in \SL(2,\Z)}\int_{\mathcal{F}_2}\int_{\SO(2,\R)}  \int_0^{\infty}\int_{\R^2} f(u_v a_t \ga g\rho'  \rho)    dv \frac{dt}{t^4}dg  d\rho'd\rho\\
   =&2\pi \int_S\sum_{\ga \in \SL(2,\Z)}\int_{\mathcal{F}_2} \int_0^{\infty}\int_{\R^2}  f(u_v a_t \ga g\rho)    dv \frac{dt}{t^4}dg d\rho\\
   =&2\pi \int_S\sum_{\ga \in \SL(2,\Z)} \int_0^{\infty}\int_{\R^2} \int_{\mathcal{F}_2}  f(u_v a_t \ga g\rho)  dg  dv \frac{dt}{t^4}d\rho\\
   =& 2\pi\int_H \int_S\int_{\mathcal{F}_2}  f(h g\rho)  dg d\rho  d\mu(h)\\
   =& \int_H \int_{\mathcal{F}_2}\int_S  f(h g\rho) 2\pi d\rho dg  d\mu(h) 
\end{align}

\vspace{3mm}
Then it follows that the measure $\nu_Y$ on $Y$ by  $$d\nu_Y:=2\pi d\rho dg,$$
satisfies \eqref{unfolding haar measure on G using section}.


Fix $Hg_0\in H\backslash G$. By identifying $H \backslash G$ with $Y$, we define a measure on $H \backslash G$ via
\begin{equation}
    d\nu_{Hg_0}(Hg):=\frac{\alpha(\sigma(Hg_0),\sigma(Hg))}{m(G/\Ga)}d\nu_{Y}(\sigma(Hg)).
\end{equation}
where $\alpha(\cdot,\cdot)$ is given in \eqref{definition and computation of alpha}. 

\vspace{3mm}
In view of Theorem 2.2 and Corollary 2.4 in \cite{Gorodnik2004DistributionOL} (duality principle), an immediate consequence of Theorem \ref{The $G$-invariance of limiting measure} is the following

\begin{corollary}
Fix $g_0\in G$. For any compactly supported $\varphi \in C_c(H\backslash G)$,
\begin{equation}
    \lim_{T\to \infty}\frac{1}{\mu(H_T)}\sum_{\ga \in \Ga_T}\varphi(Hg_0.\ga)= \int_{H\backslash G}\varphi(Hg)d\nu_{Hg_0}(Hg).
\end{equation}
\end{corollary}



Now we would like to replace the normalization factor $\mu(H_T)$ by $\# \Ga_T$.

By Theorem 1.7 of \cite{GorodnikNevo2012},
\begin{equation}
    \lim_{T\to \infty}\frac{m(G_T)}{\# \Ga_T}=m(G/\Ga).
\end{equation}

On the other-hand, by Lemma 6 of \cite{Gorodnik2003LatticeAO} (we choose the standard parameterization so that the volumes of $\SO(2,\R)$ and $\SO(3,\R)$ are $2\pi$ and $8\pi^2$ respectively),
\begin{equation}
    m(G_T)\sim 8\pi^2\cdot \frac{\pi^2}{24}T^6=\frac{\pi^4}{3}T^6.
\end{equation}
Hence, by Proposition \ref{computation for H and V},
\begin{equation}
    c_{\Ga}:=\lim_{T\to \infty}\frac{\mu(H_{T})}{\#\Ga_T}=\lim_{T\to \infty}\frac{\mu(H_{T})}{m(G_T)}\frac{m(G_T)}{\#\Ga_T}=
    \frac{m(G/\Ga)}{2\pi^3}\sum_{\ga \in \SL(2,\Z)}\frac{1}{\|\ga \|^4},
\end{equation}
which shows that $$\lim_{T\to \infty}\frac{1}{\#\Ga_T}\sum_{\ga \in \Ga_T}\varphi(Hg_0.\ga)= c_\Ga\int_{H\backslash G}\varphi(Hg)d\nu_{Hg_0}(Hg).$$

Now we compute $\nu_{Hg_0}(H\backslash G)$.  For $g_0,y \in Y$, we consider the decompositions 
\begin{equation} \label{decompositions of g_0 and y in Y}
    g_0:=\begin{bmatrix} G_1 & 0 \\0 & 1 \end{bmatrix}\rho_0, y:=\begin{bmatrix} H_1 & 0 \\0 & 1 \end{bmatrix} \rho,
\end{equation}
where $G_1,G_2\in\mathcal{F}_2$ and $\rho_0,\rho\in\SO(3,\R).$ Hence $\alpha$ in \eqref{definition and computation of alpha} takes a more simpler form (note $g_0^{-1}:=\rho_0^{-1}\begin{bmatrix} G_1^{-1} & 0 \\0 & 1 \end{bmatrix}$)
\begin{align*}
    \alpha(g_0,y)
    =& \frac{\sum_{\ga\in \SL(2,\Z)}\frac{1}{\|G_1^{-1}\ga H_1\|^4}}{\sum_{\ga\in \SL(2,\Z)}\frac{1}{\|\ga \|^4}}.
  \end{align*}

Using the invariance of the $\SO(3,\R)$ invariance of the Hilbert-Schmidt norm, we compute 
\begin{align*}
    &m(G/\Ga)\nu_{Hg_0}(H\backslash G)\\
    =&\int_{Y}\alpha(g_0,y)d\nu_{Y}\\
    =&\frac{8\pi^2}{\sum_{\ga\in \SL(2,\Z)}\frac{1}{\|\ga \|^4}} \sum_{\ga\in \SL(2,\Z)}\int_{\SL(2,\Z)\backslash \SL(2,\R)/\SO(2,\R)}\frac{1}{\|G_1^{-1}\ga H_1\|^4} dH_1 \\
    =&\frac{4\pi}{\sum_{\ga\in \SL(2,\Z)}\frac{1}{\|\ga \|^4}} \int_{\SL(2,\R)}\frac{1}{\|G_1^{-1}g\|^4}  dg\\
    =&\frac{4\pi}{\sum_{\ga\in \SL(2,\Z)}\frac{1}{\|\ga \|^4}} \int_{\SL(2,\R)}\frac{1}{\|g\|^4}  dg.
\end{align*}

To compute the above integral, we consider the $\SL(2,\R)$ Iwasawa decomposition,
\begin{equation*}
    g=\begin{bmatrix}
    1 & x\\
    0 & 1
    \end{bmatrix}\begin{bmatrix}
    a & 0\\
    0 & \frac{1}{a}
    \end{bmatrix}k.
\end{equation*}

Noticing that  $k$ preserves the Euclidean $2$-norm, the integral above is equal to

\begin{align*}
    \int_{\SL(2,\R)}\frac{1}{\|g\|^4}  dg=&2\pi\int_{\R_+}\int_{\R} \frac{1}{\left( a^2+\frac{1}{ a^2}+\frac{x^2}{a^2} \right)^2} \frac{dxda}{a^3}\\
    =&\pi\int_{\R_+}\int_{\R} \frac{1}{\left( y+\frac{1}{ y}+\frac{x^2}{y} \right)^2} \frac{dxdy}{y^2}\tag{let $a=\sqrt{y}$}\\
    =&\pi\int_{\R_+}\int_{\R} \frac{1}{\left(1+x^2+y^2 \right)^2} dxdy\\
    =&\frac{\pi^2}{2}
\end{align*}

Therefore, 
\begin{equation}
    \nu_{Hg_0}(H\backslash G)=\frac{2\pi^3}{m(G/\Ga)\sum_{\ga\in \SL(2,\Z)}\frac{1}{\|\ga\|^4}}
\end{equation}

and thus 
$$c_{\Gamma}\nu_{Hg_0}(H\backslash G)=1.$$
Therefore, $\tilde\nu_{Hg_0}:=c_{\Ga}\nu_{Hg_0}$ is a probability measure, and we conclude 
\begin{equation}
    \lim_{T\to \infty}\frac{1}{\# \Ga_T}\sum_{\ga \in \Ga_T}\varphi(Hg_0.\ga)= \int_{H\backslash G}\varphi(Hg)d\tilde \nu_{Hg_0}(g).
\end{equation}

\vspace{5mm}
We express $\tilde\nu_{Hg_0}$ as $$d\tilde\nu_{Hg_0}=\Phi_{Hg_0}(y)dg d\rho,$$ where $$\Phi_{Hg_0}(y):=\sum_{\ga\in \GL(2,\Z)}\frac{1}{\|G_1^{-1}\ga H_1\|^4},$$  $\Phi_{Hg_0}(y)dg$ is a probability measure on $\mathcal{F}_2$, and $d\rho$ is a probability measure on $\SO(3,\R)$. Now we interpret the coefficient $\Phi_{Hg_0}(y)$
in terms of rank-two discrete subgroups in $\R^3$. Let $(\La_0,w_0)$ and $(\La,w)$ denote the oriented rank-two discrete subgroups of $\R^3$ corresponding to $g_0$ and $y$, respectively. Let $e_1,e_2$ be the canonical basis of $\R^2$, written in terms of row vectors:

Then by \eqref{decompositions of g_0 and y in Y} it follows that  $$\mathscr{B}_0:=\{(e_1G_1,0)\rho_0,(e_2G_1,0)\rho_0 \},$$ and $$\mathscr{B}:=\{(e_1\ga H_1,0)\rho,(e_2\ga H_1,0)\rho\}$$ form  $\Z$-bases of $\La_0$ and $\La$, respectively.

Consider the following linear maps between two-dimensional subspaces of $\R^3$:
\begin{equation}
    T_{\mathscr{B}_0}: \Span_{\R}\{(e_1,0),(e_2,0)\} \to \Span_{\R}\{\mathscr{B}_0\}, (e_i,0)\mapsto (e_iG_1,0)\rho_0, i=1,2
\end{equation}
and
\begin{equation}
    T_{\mathscr{B}}: \Span_{\R}\{(e_1,0),(e_2,0)\} \to \Span_{\R}\{\mathscr{B}\}, (e_i,0)\mapsto (e_i\ga H_1,0)\rho, i=1,2.
\end{equation}

Recall from the introduction (see \eqref{defining hilbert-schmits for op. from hyp. to hyp.}) the Hilbert-Schmidt norm   $\|T\|_{\text{HS}}$ of an operator $T$ from the two-dimensional subspace $U\subset \R^3$ to the two-dimensional subspace $V\subset \R^3$, which is computed by choosing an orthonormal basis $\{u_1,u_2\}$ of $U$ and evaluating 
\begin{equation}
    \|T\|^2_{\text{HS}}:=\|Tu_1\|^2+\|Tu_2\|^2,
\end{equation}
where the norm on the right hand side is the usual Euclidean norm on $\R^3$. Using the orthonormal basis $\{(e_1,0)\rho_0 , (e_2,0)\rho_0\}$ of $\Span_\R\{\mathcal{B}_0\}$, we get $$\|T_{\mathscr{B}}\circ T_{\mathscr{B}_0}^{-1}\|_{\text{HS}}^2=\|G_1^{-1}\ga H_1\|^2,$$
which shows
\begin{equation}\label{intrinsic interpretation of the sum}
    \Phi_{g_0}(y)=\sum_{\ga\in \GL(2,\Z)}\frac{1}{\|G_1^{-1}\ga H_1\|^4}=\sum_{\Span_{\Z}\mathscr{B}=\La}\frac{1}{\|T_{\mathscr{B}}\circ T_{\mathscr{B}_0}^{-1}\|_{\text{HS}}^4}=\Psi_{\La_0}(\La),
\end{equation}

where $\Psi_{\La_0}(\La)$ defined in \eqref{defining psi_La_0}.
It's now straight forward to verify that the measure $\tilde\nu_{x_0}$ for $x_0=(\Span_\Z\{e_1,e_2\},e_3)\cdot g_0$ defined before Theorem \ref{equidistribution result on  G mod H} is identified with $\tilde\nu_{H_{g_0}}$:


Let $X_2=\SL(2,\Z)\backslash\SL(2,\R)$. For any $f\in C_c(X_{2,3})$, we have
\begin{align}
    \tilde\nu_{x_0}(f)
    :=&\int_{\mathbb S^2}\int_{\pi_{\perp}^{-1}(w)}f((\La,w),\rho_w) \Phi_{\La_0}(\La)d((\rho_w)_*\mu_{e_3})(\La)d\mu_{\mathbb S^2}(w)\\
    =&\int_{\mathbb S^2}\int_{\pi_{\perp}^{-1}(e_3)}f((\La,e_3),\rho_w) \Psi_{\La_0}(\La)d\mu_{e_3}(\La)d\mu_{\mathbb S^2}(w)\\
    =&\int_{\mathbb S^2}\int_{X_2}f(\Z^2\times 0,\eta\rho_w) \Psi_{\La_0}(\Z^2\times 0)d\eta d\mu_{\mathbb S^2}(w)\\
    =&\frac{1}{\vol(\SO(2,\R))}\int_{\mathbb S^2}\int_{\SO(2,\R)}\int_{X_2}f(\Z^2\times 0,\eta k \rho_w) \Psi_{\La_0}(\Z^2\times 0)d\eta dk d\mu_{\mathbb S^2}(w)\\
    =&\frac{1}{\vol(\SO(2,\R))}\int_{ \SO(3,\R)}\int_{X_2}f(\Z^2\times 0,\eta\rho) \Psi_{\La_0}(\Z^2\times 0)d\eta d\rho\\ 
    =&\frac{1}{\vol(\SO(2,\R))}\int_{ \SO(3,\R)}\int_{\mathcal F_2}\int_{\SO(2,\R)}f(\Z^2\times 0,\eta\rho k) \Psi_{\La_0}(\Z^2\times 0)d\eta d\rho dk\\
    =&\int_{ \SO(3,\R)}\int_{\mathcal F_2}f(\Z^2\times 0,\eta\rho) \Psi_{\La_0}(\Z^2\times 0)d\eta d\rho
\end{align}

The last line is the same as $\int_{H\backslash G}\varphi(Hg)d\tilde\nu_{Hg_0}(g)=\int_{H\backslash G}\varphi(Hg)d\tilde c_{\Ga}\nu_{Hg_0}(g)$ (note that both $\tilde \nu_{x_0}$ and $\tilde \nu_{Hg_0}$ are probability measures).
This proves Theorem \ref{equidistribution result on  G mod H}. 


\begin{appendices}

\section{Two copies of $\SL(2,\R)$ generates $\SL(3,\R)$}

\begin{lemma}\label{sl2 generating sl3}
Let 
\[
G_{1}:=\left\{\begin{bmatrix}
a & 0 & b\\
0 & 1 & 0\\
c & 0 & d
\end{bmatrix} : ad-bc=1\right\} ,
\]
and 
\[
G_{2}:=\left\{ \begin{bmatrix}
1 & 0 & 0\\
0 & a & b\\
0 & c & d
\end{bmatrix}: ad-bc=1\right\} .
\]
Let $G$ be a closed subgroup of $\SL(3,\R)$ which includes
$G_{1}$ and $G_{2}$. Then we claim that $G=\SL(3,\R)$.
\end{lemma}

\begin{proof}

Let us first show that 
\[
G\supseteq G_{3}:=\left\{ \begin{bmatrix}
a & b & 0\\
c & d & 0\\
0 & 0 & 1
\end{bmatrix}: ad-bc=1\right\}. 
\]

Observe that
\begin{align*}
&\begin{bmatrix}
1 & 0 & 0\\
0 & 1 & 0\\
0 & -s_{2} & 1
\end{bmatrix}\begin{bmatrix}
1 & 0 & s_{1}\\
0 & 1 & 0\\
0 & 0 & 1
\end{bmatrix}  \begin{bmatrix}
1 & 0 & 0\\
0 & 1 & 0\\
0 & s_{2} & 1
\end{bmatrix}\begin{bmatrix}
1 & 0 & -s_{1}\\
0 & 1 & 0\\
0 & 0 & 1
\end{bmatrix}\\
=&\begin{bmatrix}
1 & 0 & s_{1}\\
0 & 1 & 0\\
0 & -s_{2} & 1
\end{bmatrix}  \begin{bmatrix}
1 & 0 & 0\\
0 & 1 & 0\\
0 & s_{2} & 1
\end{bmatrix}\begin{bmatrix}
1 & 0 & -s_{1}\\
0 & 1 & 0\\
0 & 0 & 1
\end{bmatrix}\\
=&\begin{bmatrix}
1 & s_{1}s_{2} & s_{1}\\
0 & 1 & 0\\
0 & 0 & 1
\end{bmatrix}  \begin{bmatrix}
1 & 0 & -s_{1}\\
0 & 1 & 0\\
0 & 0 & 1
\end{bmatrix}\\
= & \begin{bmatrix}
1 & s_{1}s_{2} & 0\\
0 & 1 & 0\\
0 & 0 & 1
\end{bmatrix}.
\end{align*}
So for all $s\in\R$, we have $\begin{bmatrix}
1 & s & 0\\
0 & 1 & 0\\
0 & 0 & 1
\end{bmatrix}\in G$.
 By applying the matrix transpose to the above, we get that for all
$r\in\R$, $\begin{bmatrix}
1 & 0 & 0\\
r & 1 & 0\\
0 & 0 & 1
\end{bmatrix}\in G$.
 Observe that 
Let $\epsilon$ be non-zero close to zero, and note that for $r=\frac{e^{t}-1}{\epsilon}$
and $s=\frac{-r}{1+r\epsilon}$ we get
\begin{align*}
\begin{bmatrix}
1 & r\\
0 & 1
\end{bmatrix}\begin{bmatrix}
1 & 0\\
\epsilon & 1
\end{bmatrix}\begin{bmatrix}
1 & s\\
0 & 1
\end{bmatrix} & =\begin{bmatrix}
1+r\epsilon & r+s+rs\epsilon\\
\epsilon & 1+s\epsilon
\end{bmatrix}\\
 & =\begin{bmatrix}
e^{t} & 0\\
\epsilon & e^{-t}
\end{bmatrix}.
\end{align*}
So, for all $\epsilon>0$ and $t\in\R$ we have 
\[
\begin{bmatrix}
e^{t} & 0 & 0\\
\epsilon & e^{-t} & 0\\
0 & 0 & 1
\end{bmatrix}\in G
\]
and since $G$ is closed, we get by taking $\epsilon\to0$ that 
\[
\begin{bmatrix}
e^{t} & 0 & 0\\
0 & e^{-t} & 0\\
0 & 0 & 1
\end{bmatrix}\in G.
\]

The above show that in fact $G_{3}\leq G$. It follows that
that the full diagonal
group and the unipotent upper and lower triangular subgroups are all contained in $G$. By using the dimension matching argument on Lie algebras, we see their product form a neighborhood of identity in $\SL(3,\R)$.

Therefore, $G=\SL(3,\R)$ follows from the fact that a neighborhood of identity of a Lie group generates the whole group.
\end{proof}

\vspace{1cm}
\section{Estimates on the number of $\SL(2,\Z)$ points in skewed balls}

Consider
\[
B_{\tau}=\left\{ g\in\SL(2,\R):\left|\left|g\right|\right|\leq\tau\right\} .
\]
From \cite{GorodnikNevo2012} Example 4.3, we have the following estimate 
\[
N(\tau):=\left|\SL(2,\Z)\cap B_{\tau}\right|=\frac{\text{vol}(B_{\tau})}{\text{vol}(\SL(2,\R)/\SL(2,\Z))}+O_{\eta}(\text{vol}(B_{\tau})^{5/6+\eta}).
\]

{We now compute $\text{vol}(B_{\tau})$. First, recall that
the Haar measure on $\SL(2,\R)$ is expressed by 
\[
\nu(f)=\int f(n_{x}a_{t}k_{\theta})\frac{1}{e^{t}}d\theta\ dt\ dx,
\]
where $n_{x}=\begin{bmatrix}
1 & x\\
 & 1
\end{bmatrix},$ $a_{t}=\begin{bmatrix}
e^{t/2} & 0\\
 & e^{-t/2}
\end{bmatrix},\ k_{\theta}=\begin{bmatrix}
\cos\theta & -\sin\theta\\
\sin\theta & \cos\theta
\end{bmatrix}$. 

\vspace{3mm}
It follows from the bi-invariance of the Haar measure that 
\[
\text{vol}(B_{\tau})=\text{vol}\left\{ g\in\SL(2,\R):\left|\left|g\right|\right|\leq\tau\right\} .
\]
and since the norm is $k_{\theta}$ invariant, we get
\[
\text{vol}\left\{ g\in\SL(2,\R):\left|\left|g\right|\right|\leq\tau\right\} =2\pi\text{vol}\left\{ \begin{bmatrix}
e^{t/2} & xe^{-t/2}\\
 & e^{-t/2}
\end{bmatrix}:\left|\left|\begin{bmatrix}
e^{t/2} & xe^{-t/2}\\
 & e^{-t/2}
\end{bmatrix}\right|\right|\leq\tau\right\} .
\]
We have 
\[
e^{t}+\left(x^{2}+1\right)e^{-t}\leq\tau^{2}\iff x^{2}\leq e^{t}\left(\tau^{2}-2\cosh t\right).
\]
By Fubini's theorem, we see that
\[
\text{vol}\left\{ \begin{bmatrix}
e^{t/2} & xe^{-t/2}\\
 & e^{-t/2}
\end{bmatrix}:\left|\left|\begin{bmatrix}
e^{t/2} & xe^{-t/2}\\
 & e^{-t/2}
\end{bmatrix}\right|\right|\leq\tau\right\}
=\int_{\cosh t\leq\tau^{2}/2}\underbrace{\frac{2\sqrt{\tau^{2}-2\cos ht}}{e^{t/2}}}_{=f(t)}dt=I_{+}+I_{-},
\]
where $I_{+}=\int_{t\geq0,\ \cosh t\leq\tau^{2}/2}f(t)dt$, $I_{-}=\int_{t\leq0,\ \cosh t\leq\tau^{2}/2}f(t)dt$. }
\begin{enumerate}
\item {Computation of $I_{+}$: The function $\cosh t$ is injective
on $\R_{\geq0}$. and its inverse $\text{\ensuremath{\cosh}}^{-1}:[1,\infty)\to\R_{\geq0}$
has the following expression
\[
\text{\ensuremath{\cosh}}^{-1}(\alpha)=\ln(\alpha+\sqrt{\alpha^{2}-1}).
\]
 We have the following identities,
\[
\sinh(\text{cosh}^{-1}(\alpha))=\sqrt{\alpha^{2}-1},
\]
We put $\alpha=\cosh t,$ then $d\alpha=\sinh t\ dt$, which gives
\[
I_{+}=\int_{\alpha\leq\tau^{2}/2}\frac{2\sqrt{\tau^{2}-2\alpha}}{\sqrt{\alpha+\sqrt{\alpha^{2}-1}}}\frac{1}{\sqrt{\alpha^{2}-1}}d\alpha.
\]
Put $\beta=\frac{\alpha}{\tau^{2}}$, then 
\[
I_{+}=\tau^{2}\int_{\beta\leq1/2}\frac{2\sqrt{\tau^{2}-2\tau^{2}\beta}}{\sqrt{\tau^{2}\beta+\sqrt{\tau^{4}\beta^{2}-1}}}\frac{1}{\sqrt{\tau^{4}\beta^{2}-1}}d\beta
\]
}
\item {Computation of $I_{-}$: The function $\cosh t$ is injective
on $\R_{\leq0}$ and its inverse $\text{\ensuremath{\cosh}}^{-1}:[1,\infty)\to\R_{\leq0}$
has the following expression
\[
\text{\ensuremath{\cosh}}^{-1}(\alpha)=-\ln(\alpha+\sqrt{\alpha^{2}-1}).
\]
 We have the following identities,
\[
\sinh(\text{cosh}^{-1}(\alpha))=-\sqrt{\alpha^{2}-1},
\]
We put $\alpha=\cosh t,$ then $d\alpha=\sinh t\ dt$, which gives,
\[
I_{+}=\int_{\alpha\leq\tau^{2}/2}2\sqrt{\tau^{2}-2\alpha}\sqrt{\alpha+\sqrt{\alpha^{2}-1}}\frac{1}{\sqrt{\alpha^{2}-1}}d\alpha.
\]
Put $\beta=\frac{\alpha}{\tau^{2}}$, then 
\[
I_{+}=\tau^{2}\int_{\beta\leq1/2}2\sqrt{\tau^{2}-2\tau^{2}\beta}\sqrt{\tau^{2}\beta+\sqrt{\tau^{4}\beta^{2}-1}}\frac{1}{\sqrt{\tau^{4}\beta^{2}-1}}d\beta.
\]
}
\end{enumerate}
{Define 
\[
I_{1}=\int_{\beta\leq1/2}\frac{2\sqrt{1-2\beta}}{\sqrt{2\beta}}\frac{1}{\sqrt{\beta^{2}-1}}d\beta
\]
and
\[
I_{2}=\int_{\beta\leq1/2}2\sqrt{2\beta}\sqrt{1-2\beta}\frac{1}{\sqrt{\beta^{2}-1}}d\beta.
\]
Then as $\tau\to\infty$, we get that 
\[
I_{\tau}\asymp\tau^{2}I_{1}+I_{2}.
\]
}

To summarize, the above shows that $\text{vol}(B_{\tau})=c'\tau^{2}+O(\frac{1}{\tau}),$
whence 
\begin{equation}
N(\tau)=c\tau^{2}+O_{\eta}(\tau^{5/3+\eta}).\label{eq:lattice points in SL_2}
\end{equation}

\end{appendices}

\printbibliography[
heading=bibintoc,
title={Bibliography}
]
\end{document}